\documentclass[a4paper,12pt,bothside]{book}
\usepackage[hmarginratio=1:1]{geometry}
\usepackage{amsthm,textcomp, graphicx,amscd,epsfig,enumerate,amsmath,amsfonts,amssymb,multirow}

\newtheorem{defi}{Definition}[section]
\newtheorem{thm}{Theorem}[section]
\newtheorem{lem}{Lemma}[section]
\newtheorem{cor}{Corollary}[section]

\newtheorem{rem}{Remark}[chapter]

\begin{document}
\newpage
\thispagestyle{empty}
\pagenumbering{roman}
\titlepage

\begin{titlepage}
 \begin{center}
\LARGE
\textbf{ INJECTIVITY IN HIGHER ORDER COMPLEX DOMAINS} 
\end{center}
\normalsize
$$\textbf{A Dissertation Submitted to  the  University of Delhi}$$
$$\textbf{in Partial Fulfilment of the Requirements}$$
$$\textbf{for the Award of the Degree of}$$
\large
$$\textbf{MASTER OF PHILOSOPHY}$$
\normalsize
$$\textbf{IN}$$
\large
$$\textbf{MATHEMATICS}$$
\large
\\
$$\textbf{By}$$
\large
$$\textbf{NAVEEN GUPTA}$$
\\
\vskip6em
\normalsize
$$\textbf{DEPARTMENT\,OF\,MATHEMATICS\,}$$
$$\textbf{UNIVERSITY\,OF\,\,DELHI}$$
$$\textbf{DELHI-110007}$$
$$\textbf{July 2014}$$
\end{titlepage}

\thispagestyle{empty}

\begin{center}{\textbf{DECLARATION}}
\end{center}

\vskip2em
\fontsize{13}{23}
\selectfont
\quad This dissertation entitled ``INJECTIVITY IN HIGHER ORDER COMPLEX DOMAINS'',  a critical survey of the work done by various authors in this area, has been prepared by me under the supervision of Dr. Sanjay Kumar, Associate Professor, Department of Mathematics, Deen Dayal Upadhyaya College, University of Delhi for the submission
to the University of Delhi as partial fulfilment for the award of the degree of \textbf{Master of Philosophy in Mathematics}.

I, hereby also declare that, to the best of my knowledge, the work presented in this dissertation, either in part or full, has not been submitted to any other University or elsewhere for the award of any degree
 or diploma.
\vskip3em
\noindent
\rightline{\textbf{(\,Naveen Gupta\,) \,}}
 \vskip9em
\noindent\baselineskip16pt
\textbf{Dr. Sanjay Kumar}            \hspace{5cm} \textbf{\,Prof. Ajay Kumar }\\
Supervisor  \& Associate Professor  \hspace{3cm}              Head \\
Department of Mathematics   \hspace{3.9cm}         Department of Mathematics \\
Deen Dayal Upadhyaya College \hspace{3.3cm}        University of Delhi\\
University of Delhi    \hspace{5.6cm}                        Delhi-110007                                                                                                                           

\baselineskip24pt
\newpage
\vskip4em
\thispagestyle{empty}

\begin{center}{\textbf{\Large{ACKNOWLEDGEMENTS}}}
\end{center}

\fontsize{12}{20}
\selectfont

\vskip1em
\quad It is a great pleasure for me to express intense gratitude to my M.Phil. dissertation advisor Dr. Sanjay Kumar for his keen interest in my topic of dissertation,
 guidance, encouragement and patience throughout the writing of the dissertation. Despite of his  busy schedule, he has given me enough time. I am indeed indebted to him.
 
I am delighted to take the opportunity to express my sincere thanks to Prof. Ajay Kumar, Head, Department of Mathematics, University of Delhi, for providing the necessary facilities without any delay in work. 

I would like to thank my parents for their constant support. I give my heartfelt thanks to my friends and seniors  for their timely help and the moral support they provided me during my dissertation work. I am specially thankful to Gopal Datt, Kushal Lalwani and Shelly Verma for their help and support. I am thankful to UGC for providing me fellowship.

I would like to thank  the staff of Central Science Library for being cooperative.
Finally, I thank the office staff of Department of Mathematics for being so helpful. 

\vskip3em \fontsize{12}{20} \selectfont \noindent
\;\;\;\;\;\;\hfill{\textbf{ (\,Naveen Gupta\,) \,}} \eject
\fontsize{12}{19.5}
\selectfont
\newpage
\newpage
\vskip4em
\thispagestyle{empty}

\begin{center}{\textbf{\Large{Preface}}}
\end{center}

\fontsize{12}{20}
\selectfont

\vskip1em
The aim of the dissertation is to study the \emph{injective} aspect of functions defined on $n$-dimensional complex domain. There are basically two topics discussed in the dissertation. First one is relation of injectivity with Jacobian and another one is about injectivity of Poisson integral. Both the problems are discussed for different classes of functions.

It is a well known result of complex analysis of several variables that \emph``{An analytic function is locally injective in a neighbourhood of a point if and only if its Jacobian does not vanish at that point}". By a complex valued harmonic function, we mean a complex valued function whose real and imaginary parts are harmonic. In 1936 ~\cite{lewy}, H. Lewy proved above mentioned result for complex valued harmonic functions. However, Lewy's theorem fails for harmonic functions on $n$-dimensional complex domains. In 1974, J. C. Wood disproved it in higher dimensions by giving a counterexample in his thesis ~\cite{wood1}. So, we study the same problem for a class of functions which is bigger than that of analytic functions but smaller than that of harmonic functions, namely, pluriharmonic functions. In 1976 ~\cite{naser}, M. Naser proved that Lewy's theorem holds true for $2$-dimensional complex domains for pluriharmonic functions.
Interestingly, for $n>2$ this problem is open for pluriharmonic functions.

By the Poisson integral formula, a real valued continuous function on a circle can be extended to a harmonic function. For a continuous complex valued function on circle, using Poisson integral formula for its real and imaginary parts, it can be extended to a complex valued harmonic function. Tibor Rad\v{o} posed a question about injectivity of this Poisson integral. In 1926, Hellmuth Kneser ~\cite{kneser} supplied a proof of this result. In 1945, ~\cite{choquet} Gustave Choquet, unaware of the work done by Rad\'{o} and Kneser discovered this result with a different proof. This theorem fails in higher dimensions. In 1996 ~\cite{laug}, R. S. Laugesen provided a counterexample for this.

The dissertation is divided into five chapters followed by bibliography at the end. Chapters 2 to 5 are embodiment of our study in the present dissertation. The chapterwise organization of the dissertation is as under.

\vspace*{.4cm}\textbf{Chapter 2} entitled  `` \textit{Introduction to several variable complex analysis}" provides an introduction to the theory of several variables. Some results of several variables are discussed with proof. A map on $n$-dimensional complex domain is said to be separately analytic if it is analytic in each variable separately. It is quite fascinating that a function which is seperatley analytic is analytic jointly, which is the content of Hartog's theorem. When dealing with several variables, study of Hartog's theorem on separate analyticity becomes necessary. So we have included rigorous proof of Hartog's theorem, following the treatment of H\"{o}rmander ~\cite{hor} and Paul Garret ~\cite{paul}. We conclude this chapter by giving some similarities and differences between one variable and several variable complex theory.

\vspace*{.4cm}In \textbf{Section 2.1}, we have included introduction to several variable complex theory. A series of lemmas are proved to finally prove Hartog's theorem on separate analyticity. Analogous of Cauchy's integral formula and some other results of one variable complex theory for functions of several variable are included.

\vspace*{.4cm}In \textbf{Section 2.2}, we have explained the similarities and differences between one variable complex theory and several variable complex theory.  

\textbf{Chapter 3} entitled `` \textit{Invertibility and Jacobian}", discusses relation of Jacobian of a function with its injectivity.  We begin the chapter by giving formal definition of Jacobian of a funciton and then relating it with derivative of the funciton.

\vspace*{.4cm}In \textbf{Section 3.1}, we have discussed relation of injectivity with Jacobian for the one dimensional case. We have discussed proof of Lewy's theorem given in ~\cite{lewy}.

\vspace*{.4cm}In \textbf{Section 3.2}, we have discussed work done by M. Naser ~\cite{naser}. In fact, Lewy's theorem comes out as a corollary to Naser's work.

\textbf{Chapter 4} entitled `` \textit{Injectivity of harmonic extensions}" discusses injectivity of Poisson integral. In this chapter some properties of Possion integral are discussed and Rad\'{o}-Kneser-Choquet theorem is proved.

\textbf{Chapter 5} entitled `` \textit{Counterexamples in higher dimensions}" includes counterexamples for results mentioned in chapters 3 and 4. R. S. Laugesen's counterexample is explained thoroughly. This example is also termed as ` tennis ball example', as the construction given in the example comes out to be as that of a tennis ball. J. C. Wood ~\cite{wood1}, in 1974 included a counterexample to Jacobian problem in his thesis unaware of the fact that such an example was not known at that time. In 1991, when told by P. Duren, he presented this counterexample in his paper ~\cite{wood}.\\

\tableofcontents
\newpage



 \normalsize
 \newpage
\newpage
\pagenumbering{arabic}
\thispagestyle{plain}
\chapter{Some preliminary concepts of theory of one variable}
\section{Complex Analysis}
By complex space, denoted by $\mathbb{C}$, we mean the space given by $\{a+ib:a,b\in{\mathbb{R}}\}$. Complex analysis is the study of functions defined on $\mathbb{C}$. Although $\mathbb{C}$ can be related with $\mathbb{R}^2$ under the mapping $(a+ib)\to (a,b)$, yet functions on $\mathbb{C}$ behaves totally different from function on the euclidean space $\mathbb{R}^2$. We will see all this in this section. Let us begin by recalling some of the definitions and introducing terminologies.
\begin{defi}
A subset $\Omega \subseteq \mathbb{C}$ is called a region if it is open and connected. Domain is another commonly used term for a region.
\end{defi}
\begin{defi} 
Let $\Omega \subset \mathbb{C}$ be an open set. A function $f:\Omega \to \mathbb{C}$ is said to be differentiable at a point $z_0\in {\Omega}$ if \\
$$\lim_{h\to 0}\frac{f(z_0+h)-f(z_0)}{h}$$ exists.
We say $f$ is differentiable on $\Omega$ if it is differentiable at each point of $\Omega$. We say $f$ is continuously differentiable if its derivative, denoted by $f'$, is continuous. 
\end{defi}
For a complex valued function $f$, we will denote real and imaginary parts of $f$ by $\Re f$ and $\Im f$ respectively. We say a function $f$ is of class $C^n$ if $f$ is $n$ times differentiable and the $n$th derivative is continuous.
\begin{defi}
A function $f$ is said to be injective if no two distinct elements has the same image under $f$. An injective or one-one function is also called univalent function.
\end{defi}
\begin{defi}
A function $f:\Omega \to \mathbb{C}$ is said to be analytic if it is continuously differentiable on $\Omega$.
\end{defi}
Let $f:\Omega \to \mathbb{C}$ be analytic.
Let us denote by the real and imaginary parts of $f$ by $u$ and $v$ respectively,
   $$u(x,y)=\Re f(x+i y)\ \mbox{and} \ 
v(x,y)=\Im f(x+i y).$$
At a point $z=x+i y\in{\Omega}$, $f$ is differentiable,
$\mbox{that is}$  $$\lim_{h\to 0}\frac{f(z_0+h)-f(z_0)}{h}$$ exists.
First assume that $h\to 0$ through real axis,
so we have $$\frac{f(z_0+h)-f(z_0)}{h}=\frac{u(x+h,y)-u(x,y)}{h} +i \frac{v(x+h,y)-v(x,y)}{h}.$$
Letting $h\to 0$, we get
$$f'(z)=u_x(x,y)+i v_x(x,y),$$ where $u_x$ and $v_x$ denote partial differentiation of $u$ with respect to x.\\
Now, assume $h\to 0$ through imaginary axis,
therefore we have $$\frac{f(z_0+h)-f(z_0)}{h}=(-i)\frac{u(x,y+h)-u(x,y)}{h}+\frac{v(x,y+h)-v(x,y)}{h}.$$
Letting $h\to 0$, we get
$$f'(z)=-i u_y(x,y)+v_y(x,y).$$
Comparing real and imaginary parts from both the equations,we get
$$u_x=v_y\ \mbox{and} \  u_y=-v_x,$$
these equations are called Cauchy Riemann equations.
Further differentiating, we get-
$$u_{xx}=v_{yx}\ \mbox{and} \ u_{yy}=-v_{xy},$$
which implies that  $$u_{xx}+u_{yy}=0.$$
A function $u$ satisfying this equation is called harmonic. Thus $u$ is harmonic.
In a similar way, $v$ is also harmonic.
We observe that real and imaginary parts of an analytic function are harmonic and satisfy Cauchy Riemann equations.
\begin{rem} Converse of the above statement is also true, which is stated as:
\end{rem}
\begin{thm}
Let $u,v:\Omega \subseteq \mathbb{C}\to \mathbb{C}$ be functions having continuous first order partial derivatives. Then $f:\Omega \to \mathbb{C}$ defined as 
$$f(z)=u(z)+i v(z)$$ is analytic if and only if $u,v$ satisfy Cauchy Riemann equations.
\end{thm}
Such $u$ and $v$ are called harmonic conjugates. It is worth mentioning that if $\Omega $ happens to be whole of $\mathbb{C}$ or some open disk, then for any $u$ harmonic on $\Omega$ there is a harmonic function $v$ on $\Omega$ such that $f=u+i v$ is analytic.
\begin{defi}
Consider power series $\sum a_n(z-a)^n$. The radius of convergence of this power series is defined as 
$$\frac{1}{R}=\lim \sup |a_n|^{\frac{1}{n}}.$$
\end{defi}
We now state some of the results of function theory of one complex variable, which we will be using frequently in the further chapters. We will omit the proofs of the results stated here. For proof one can see ~\cite{ahlfors1} and ~\cite{conway1}.
\begin{thm}
Let $f$ be analytic in $B(a,R)$, then $$f(z)=\sum_{n=0}^{\infty}a_n(z-a)^n \  \mbox{for} \  |z-a|<R,$$ where $\displaystyle{a_n=\frac{1}{n!}f^{(n)}(a)}$ and this series has radius of convergence greater than equal to $R$.
\end{thm}
\begin{thm}
If $f:\Omega \to \mathbb{C}$ is analytic. Then $f$ is infinitely differentiable.
\end{thm}
\begin{rem}
This is one of the most interesting aspects of functions analytic on domains of $\mathbb{C}$. In $\mathbb{R}^n\ (n\geq 2)$, such a result does not hold. In fact, a function which is once differentiable may not even be twice differentiable. But in case of a function defined on a domain of $\mathbb{C}$, if $f$ is once differentiable, it becomes infinitely differentiable.
\end{rem}
\begin{thm}
If $f:\Omega \to \mathbb{C}$ is analytic and not identically zero. Then the set
$\displaystyle{\{z:f(z)=0\}}$ does not have a limit point in $\Omega$.
\end{thm}
\begin{cor}
Zeroes of a non-constant analytic function are isolated.
\end{cor}
\begin{rem}
This result fails to hold in several complex variable theorey. We will see this in next chapter.
\end{rem}
\begin{defi}
A curve $\gamma$ in $\mathbb{C}$ is a continuous function
 $$\gamma :[a,b]\to \mathbb{C}.$$
$\gamma$ is said to be closed if $\gamma (a)=\gamma (b)$.
For a curve $\gamma$, image set of $\gamma$ denoted by, $\{\gamma\}$, is called trace of $\gamma$.
\end{defi}
\begin{defi}
For a function $\gamma:[a,b]\to \mathbb{C}$, variation of $\gamma$ with respect to a partition $P=\{a=t_0<t_1\ldots , t_n=b\}$ of $[a,b]$ is defined as 
$$v (\gamma;P)=\sum_{k=1}^{n}|\gamma (t_k)-\gamma (t_{k-1})|.$$
The total variation of $\gamma$ is defined as
$$V(\gamma)=\sup \{v(\gamma ;P): \mbox{P is a partition of P}\}.$$

\end{defi}
\begin{defi}
A curve $\gamma$ is said to be of bounded variation if $V(\gamma)\leq M$ for some $M>0$.
\end{defi}
\begin{defi}
A curve with bounded variation is called a rectifiable curve.
\end{defi}
\begin{defi}
Let $\gamma$ be a closed rectifiable curve in $\mathbb{C}$. Then for $z_0 \notin \{\gamma \}$, index of $\gamma$ with respect to $z_0$, denoted by $n(\gamma ;z_0)$, is defined as 
$$n(\gamma ;z_0)=\frac{1}{2\pi i}\int_{\gamma}(z-z_0)^{-1}dz.$$
\end{defi}
\begin{defi}
Let $U$ be a subset of $\mathbb{C}$. A maximal connected subset of $U$ is called component of $U$. 

\end{defi}
\begin{thm}
Let $\gamma$ be a closed rectifiable curve in a region $\Omega \subseteq \mathbb{C}$. Then $n(\gamma ;z_0)$ is constant for $z_0$ belonging to a component of $\Omega \setminus \{\gamma \}$. And $n(\gamma ;z_0)=0$ for $z_0$ belonging to an unbounded component of $\Omega$.
\end{thm}
\begin{thm}
Let $\Omega \subseteq \mathbb{C}$ be a region and let $f$ be analytic on $\Omega $ with $a_1,a_2,\ldots ,a_n$ zeroes of $f$ (counted according to multiplicity). If $\gamma $ is a closed rectifiable curve in $\Omega$ not passing through any of the $a_i's$ and if $\gamma$ is homologus to zero, then
$$\frac{1}{2\pi i}\int_{\gamma}\frac{f'(z)}{f(z)}dz=\sum_{i=1}^{n}n\left(\gamma ;a_i\right).$$
\end{thm}
\begin{thm}
Suppose $f$ is analytic in $B(a,R)$ and let $\alpha =f(a)$. If $f(z)-\alpha $ has a zero of order $m$ at $z=a$, then there is an $\epsilon >0$ and a $\delta >0$ such that for 
$|\zeta -\alpha |<\delta$, the equation $f(z)=\zeta$ has exactly $m$ simple roots in $B(a,\epsilon)$. 
\end{thm}
Now, we state a very beautiful result due to Goursat, which says that we do not need to impose the condition of $f'$ to be continuous in the definition of analytic functions.
\begin{thm}
Let $\Omega \subseteq \mathbb{C}$ be open and $f:\Omega \to \mathbb{C}$ be differentiable. Then $f$ is analytic in $\Omega$.
\end{thm}

\section{Harmonic Mappings}
Harmonic functions play a vital role in the field of physics and in solving partial differential equations. In complex analysis as well, class of harmonic functions are one the most studied and useful functions. In the case of function theory of one variable, an analytic function has a distinguishing property. In this case real and imaginary parts of an analytic function are harmonic functions which are harmonic conjugate to each other and related via Cauchy Riemann equations. This has vast implications in physics, namely if one is considered as a potential, then the other is the flow which the potential induces. A harmonic function is uniquely determined by its boundary values. We can construct a harmonic function with prescribed boundary values and determine its conjugate locally uniquely upto an additive constant. Thus theoretically, they give us a way to construct analytic function on certain kind of domains (simply connected). This makes it easy to construct analytic functions of one complex variable. Another importance of harmonic mappings in complex theory of one variable is that real and imaginary parts of an analytic function are harmonic. Thus in order to study some aspects of analytic functions it suffices to study the harmonic mappings. And a vast number of interesting results for analytic functions follows from the study of harmonic mappings. The main properties of analytic functions can be deduced by studying the harmonic functions.
\begin{defi}$(\textbf{\mbox{Harmonic mapping}})$
Let $u:G\subseteq \mathbb{C} \rightarrow \mathbb{R}$ be of class $C^2$. Then $u$ is said to be real harmonic if it satisfies the Laplace's equation
$$\displaystyle{\frac{\partial ^2u}{\partial x^2}+\frac{\partial ^2u}{\partial y^2}=0.}$$
A mapping $u:G\subseteq \mathbb{C} \rightarrow \mathbb{C}$ is said to be complex harmonic if its real and complex parts are real harmonic.
$\mbox{That is} \  \Re u , \Im u:G\subseteq \mathbb{C} \rightarrow \mathbb{R}$ are real harmonic. Equivalently, a complex valued function $f$ defined on $\mathbb{C}$ is complex harmonic if and only if ${\displaystyle\frac{\partial^2 f}{\partial z \partial \bar{z}}=0.}$ 
\end{defi}

\begin{rem}
If $f=u+iv:G\subseteq \mathbb{C} \rightarrow \mathbb{C}$ is analytic funtion, then $u$ and $v$ satisfy Cauchy Riemann equations:
$$\frac{\partial u}{\partial x}=\frac{\partial v}{\partial y} \ 
\mbox{and} \  \frac{\partial u}{\partial y}=-\frac{\partial v}{\partial x}.$$
Again differentiating partially we get that $$\frac{\partial ^2u}{\partial x^2}=\frac{\partial ^2v}{\partial x \partial y}\ \mbox{and} \ \frac{\partial ^2u}{\partial y^2}=-\frac{\partial ^2v}{\partial y \partial x},$$
which in turn implies that $$\displaystyle{\frac{\partial ^2u}{\partial x^2}+\frac{\partial ^2u}{\partial y^2}=0.}$$
Thus $u$ is real harmonic, similarly, $v$ is real harmonic and
hence $f$ is complex harmonic.
Thus we get that every analytic function is complex harmonic.
\end{rem}
It is noteworthy that for a function $f=u+iv:G\subseteq \mathbb{C} \rightarrow \mathbb{C}$ to be harmonic, we require that $u$ and $v$ to satisfy Laplace's equation. On the other hand, for $f$ to be analytic, we require real and imaginary parts of $f$ to satisfy Cauchy Riemann equations.
Already, we have seen that Cauchy Riemann equations implies Laplace's equation.
Hence class of complex harmonic functions contains the class of analytic functions.
\begin{rem}
For a harmonic function $u$ defined on whole of $\mathbb{C}$ or on some open disk, there exists a harmonic function $v$ defined on the domain such that $f=u+iv$ is analytic.
Such $v$ is called harmonic conjugate of $u$.
\end{rem}
\begin{rem}
A complex harmonic map need not be analytic.
\end{rem}
\textbf{Example of a complex harmonic map, which is not analytic}\\
Consider map $u:G=\mathbb{C} \setminus \{0\} \rightarrow \mathbb{C}$ defined as 
$$u(z)=log |z|
=\frac{1}{2}log (x^2+y^2), \ z=x+iy,$$
then $u$ satisfies Laplace's equation and it does not have any harmonic conjugate.
So, $u+iv$ is not harmonic for any choice of harmonic map $v$.\\
Thus, in particular if we take $v$ to be defined on $G$ as
$$v(z) =a\ \mbox{for all} \  z\in{G},$$ where a is any non-zero constant.
Then $v$ is harmonic.
Thus $f=u+iv$ is complex harmonic but not analytic.
\begin{rem}
Composition of a harmonic function with an analytic function is harmonic. More precisely, harmonic of an analytic function is harmonic. But composition of two harmonic functions may not be harmonic. For example consider map given by $f(z)=z+\bar{z}$ and $g(z)=z^2$, then both $f$ and $g$ are harmonic but composition given by $g\circ f(z)=(z+\bar{z})^2$ is not harmonic. Note that this example also shows that analytic of a harmonic map may not be harmonic.
\end{rem}
\begin{rem}
Product of two analytic functions is analytic, this is no longer true for harmonic functions.
For example, consider $u$ defined as $$u(x+iy)=x$$ is harmonic.
But $u.u(z)=x^2$, is not harmonic.
\end{rem}
\begin{defi}
Let $G\subseteq \mathbb{C}$ be open and $u:G\to \mathbb{R}$ be continuous. Then $u$ is said to have mean value property (MVP) if whenever $ \overline{B(a,r)}\subseteq G,$ $$u(a)=\frac{1}{2\pi}\int_0^{2\pi}u(a+re^{i t}) dt.$$
\end{defi}
\begin{lem}
Let $G\subseteq \mathbb{C}$ be open and $u:G\to \mathbb{R}$ be harmonic. Then whenever $ \overline{B(a,r)}\subseteq G,$ $$u(a)=\frac{1}{2\pi}\int_0^{2\pi}u(a+re^{i t}) dt.$$
\end{lem}
Above mentioned lemma says that a harmonic function has mean value property. The converse of this statement is also true, which is stated as:
\begin{lem}
Let $G\subseteq \mathbb{C}$ be open and $u:G\to \mathbb{R}$ be continuous. Suppose that $u$ has mean value property. Then $u$ is harmonic in $G$.
\end{lem}

\chapter{ Introduction to Several variable complex analysis}
The theory of functions of several complex variables is a branch of mathematics which deals with the complex valued functions defined on a domain consisting of complex $n-$tuples. The general theory of functions of several complex variables came into the light far later than the theory of functions of one variable. Many examples of such functions were available in the nineteenth century mathematics, for example theta function and any function of one complex variable that depends on some complex parameter. But theory did not become full fledged area in mathematical analysis, since its importance was not uncovered till that time. The foundation of the subject was laid by Weierstrass in the nineteenth century and with the work of Friedrich Hartogs and Kiyoshi Oka in the 1930s, a general theory began to emerge. Hartogs proved some basic results including Hartogs result on separate analyticiy.
After 1945 important work in france, mainly seminar of Henri Cartan and in germany, work of Hans Grauert and Reinhold Remmert, quickly changed the picture of the theory and the theory began to bloom.\\

The theory of functions of several complex variables is quite difficult to grasp compared to the theory of one complex variable. Main reason is the freedom that the domains have in $\mathbb{C}^n$ due to the increase in the dimension. Another reason is that both the real and imaginary parts of an analytic function are now pluriharmonic functions. Which means more restriction on real and imaginary parts than merely being harmonic. Unlike theory of one variable, where we can start by constructing harmonic functions with prescribed values on the boundary of the domain, we can not always construct a pluriharmonic function with prescribed values on certain portion of the boundary. In some of the cases we can construct such functions, but not always. Therefore, it is difficult to construct analytic functions of several complex variables. Since function theory of one complex variable generally proceeds by constructing analytic functions, we can not simply use the one-variable approach to the case of several complex variables.
\section{Properties of several variables}
Let $n\in {\mathbb{N}}$, by $\mathbb{C}^n$ we denote $n$-dimensional complex space. $\mathbb{C}^n$ has elements of the form $(z_1,z_2,\ldots , z_n):z_i \in{\mathbb{C}} \ (1\leq i\leq n)$.
\begin{defi}
Let $a\in {C}^n$ and $r_j>0$, $j=1, 2, \ldots, n$, a polydisk centered at a with polyradius $r=(r_j)$, denoted by $P(a,r)$, is the subset of $\mathbb{C}^n $ defined as 
$$P(a,r):=\{z\in{\mathbb{C}^n}:\left|z_j-a_j\right|<r_j, \ \mbox{for }\ j=1,2,\ldots, n\}.$$ We denote this polydisk by $P(a,r)$. 
\end{defi}
Note  that this polydisk can be written in the following way
\begin{align*}
P(a,r)
&=\{z\in{\mathbb{C}^n}:\left|z_j-a_j\right|<r_j, \ 1\leq j\leq n \}\\
&=\prod_{1\leq j \leq n} D_j,\  \mbox{where} \  D_j=\{z:|z_j-a_j|<r_j\}.
\end{align*}
Closure of $P(a,r)$ is given by $$\overline{P(a,r)}=\{z\in{\mathbb{C}^n}:|z_j-a_j|\leq r_j, \ \mbox{for }\ j=1,2,\ldots, n\},$$
 and it is called closed polydisk centered at $a$ with polyradius $r$. In case when $r_j=r$ for $j=1,2,\ldots, n$, we call $P(a,r)$ polydisk centered at $a$ with radius $r$.
\begin{rem}
A multi index is of the form $ \alpha =\left({\alpha}_1,\ldots, \alpha_n\right)$, that is, a multi index is an element of $\mathbb{N}^n .$ By $z^v$ we will denote $z_1^{v_1}z_2^{v_2}\ldots z_n^{v_n}. $ We will write $\displaystyle{\frac{\partial^{\alpha}}{\partial z^\alpha}}$ for $\displaystyle{\frac{\partial^{\alpha_1}}{\partial z_1^{\alpha_1}}\ldots \frac{\partial^{\alpha_n}}{\partial z_n^{\alpha_n}}}$. Similarly, we will write $\alpha !=\alpha_1 !\ldots \alpha_n !$ and $|\alpha|=\alpha_1+\ldots +\alpha_n$.
\end{rem}
\begin{defi}
By a domain in $\mathbb{C}^n$ we mean an open and connected subset of $\mathbb{C}^n$. Let $D_j, \ (1\leq j\leq n)$ be domains in $\mathbb{C}$, and consider the subset $D$ of $\mathbb{C}^n$ defined as 
\begin{align*}
D&=\prod_{j=1}^n D_j\\
&=\{z=(z_1,z_2,\ldots, z_n)\in{\mathbb{C}^n}:z_j\in{D_j}, j=1,2,\ldots, n\}.
\end{align*}
D is called the product domain in $\mathbb{C}^n$.
\end{defi}
\begin{defi}
A function $f:\Omega \subseteq \mathbb{C}^n \to \mathbb{C}$ is said to be holomorphic if for each $z^0\in {\Omega}$, there is $r=r(z^0)>0$ such that $ \overline{P}(z^0,r)\subseteq \Omega$ and $f$ can be written as
$$f(z)=\sum_v a_v(z-z^0)^v ,\ \mbox{for all} \  z\in {P(z^0,r)},$$ where $v$ denotes multi index.
\end{defi}
\begin{defi}
A complex polynomial in $\mathbb{C}^n$ is of the form $$P(z)=\sum_{v}c_vz^v,$$ where $v$ is a multi index and $c_v\ne 0$ for finitely many $c_v\in{\mathbb{C}}.$ The degree of $P$ is defined as $$degP=\max \{v_1+\ldots +v_n:v \ \mbox{is a multi index}, c_v\neq 0\}.$$ 
\end{defi}
\textbf{Example}:
Consider complex polynomial on $\mathbb{C}^n$ given by $$P(z_1,z_2,\ldots, z_n)=z_1,$$ the degree of $P$ is 1. Consider another complex polynomial on $\mathbb{C}^6$ defined by 
$$P_1(z)=z_1^3+z_3^4z_5^2+z_4^5+z_2^1z_3^2z_4^2,$$ on $\mathbb{C}^6$. The degree of $P_1$ is $ =\max \{v_1+\ldots +v_n:v\ \mbox{is a multi index}, c_v\neq 0\}=6.$
Note that every complex polynomial in $\mathbb{C}^n$ is analytic on whole of $\mathbb{C}^n$.
\begin{defi}
A complex polynomial $P(z)=\displaystyle{\sum_{v}c_vz^v}$ is homogeneous if degree of each of the term in the polynomial is same. In other words, the set $\{v_1+\ldots +v_n:v\in{\mathbb{C}^n}, c_v\neq 0\}$ is singleton. 
\end{defi}
\textbf{Example}: The polynomial $P$ in the previous example is homogeneous but $P_2$ is not homogeneous. Another example of homogeneous polynomial can be given as $$P_2(z)=z_1^2z_3+z_1z_2z_3+z_2^3+z_3^3$$ on $\mathbb{C}^4.$
\begin{defi}
A function $f:\Omega \subseteq \mathbb{C}^n \to \mathbb{C}$ is said to be holomorphic in each variable separately if for each $i=1,2,\ldots ,n$ and for each fixed $z_1,z_2,\ldots ,z_{i-1},z_{i+1},\ldots ,$ $z_n$, the function $g_i$ defined on $\Omega _i$ as
$$g_i(z)=f(z_1,z_2,\ldots ,z_{i-1},z,z_{i+1},\ldots ,z_n)$$ is holomorphic in one variable sense,
where $\Omega _i=\{z\in{\mathbb{C}:(z_1,z_2,\ldots ,z_{i-1},z,z_{i+1},\ldots ,z_n)}\in{\Omega}\}$.
\end{defi}
\begin{defi}
Let $G\subseteq \mathbb{C}$ be a region and $u:G\to \mathbb{R}$ be continuous. Then $u$ is said to be subharmonic in G if  $$u(a)\leq \frac{1}{2\pi}\int_{o}^{2\pi}u(a+re^{i t}) dt, \ \mbox{whenever} \ \overline{B(a,r)}\subseteq G,$$ and $u$ is said to be superharmonic if $$u(a)\geq \frac{1}{2\pi}\int_{o}^{2\pi}u(a+re^{i t}) dt, \ \mbox{whenever} \ \overline{B(a,r)}\subseteq G.$$
\end{defi}
\begin{rem}
Let $G\subseteq \mathbb{C}$ be open and $u:G\to \mathbb{R}$ be harmonic. Since a harmonic function has mean value property, therefore we have $$u(a)\leq \frac{1}{2\pi}\int_{o}^{2\pi}u(a+re^{i t}) dt,\  \mbox{whenever} \   \overline{B(a,r)}\subseteq G .$$ Thus every harmonic function is subharmonic and superharmonic. \\In fact, $u:G\subseteq \mathbb{C}\to \mathbb{R}$ is harmonic if and only if $u$ is subharmonic and superharmonic.
\end{rem}
\begin{lem}
$(\textbf{\mbox{Osgood's Lemma}} )$Let $f$ be a complex valued function defined on an open subset $U$ of $\mathbb{C}^n$. Suppose that $f$ is analytic in each variable separately and $f$ is jointly continuous in $U$. Then $f$ is jointly analytic on $U$.
\begin{proof}
Consider a polydisk $\displaystyle{P=\prod_{1\leq i \leq n}D_i }$ such that $\overline{P} \subseteq U$. Since $f$ is analytic separately, therefore  for fixed $z_2,z_3,\ldots, z_n$, the map $g$ defined as 
$$g(z)=f(z,z_2,z_3,\ldots ,z_n) ;\ z\in {D_1}$$ is analytic.\\
Therefore, by Cauchy integral formula, for $z_1\in{C_1}$, where $C_1$ is a circle bounding $z_1$,\\
we have,
\begin{align*}
g(z_1)&=\frac{1}{2\pi i }\int_{C_1}\frac{g(\zeta _1)}{\zeta_1-z_1}d\zeta_1\\
&=\frac{1}{2\pi i }\int_{C_1}\frac{f(\zeta _1,z_2,z_3,\ldots, z_n)}{\zeta_1-z_1}d\zeta_1.
\end{align*}
Now, for fixed $z_1,z_3,\ldots, z_n,$ the map defined by 
$$h(z)=f(z_1,z,z_3,\ldots, z_n)$$ is analytic, and therefore by Cauchy integral formula, for $z_2\in{C_2}$, where $C_2$ is a circle bounding $z_2$,\\
we have,
\begin{align*}
h(z_2)&=\frac{1}{2\pi i }\int_{C_2}\frac{h(\zeta)}{\zeta_2-z_2}d\zeta_2\\
&=\frac{1}{2\pi i }\int_{C_2}\frac{f(z_1, \zeta _2,z_3,\ldots, z_n)}{\zeta_2-z_2}d\zeta_2.
\end{align*}
So we get that,
\begin{align*}
f(z_1,z_2,z_3,\ldots, z_n)&=\frac{1}{2\pi i }\int_{C_1}\frac{f(\zeta _1,z_2,z_3,\ldots, z_n)}{\zeta_1-z_1}d\zeta_1\\
&=\frac{1}{2\pi i }\int_{C_1}\frac{1}{\zeta_1-z_1}d\zeta_1 \frac{1}{2\pi i }\int_{C_2}\frac{f(\zeta_1, \zeta _2,z_3,\ldots, z_n)}{\zeta_2-z_2}d\zeta_2.
\end{align*}
Proceeding like this, we obtain
\begin{align*}
f(z_1,z_2,z_3,\ldots, z_n)&=\frac{1}{2\pi i }\int_{C_1}\frac{1}{\zeta_1-z_1}d\zeta_1 \frac{1}{2\pi i }\int_{C_2}\frac{f(\zeta_1, \zeta _2,z_3,\ldots, z_n)}{\zeta_2-z_2}d\zeta_2\\
& \ \ \ \ \ \ldots \frac{1}{2\pi i }\int_{C_n}\frac{f(\zeta_1, \zeta _2,\zeta_3,\ldots, \zeta_n)}{\zeta_n-z_n}d\zeta_n\\
&=\frac{1}{2\pi i }^n\int_{C_1} \int_{C_1}\ldots \int_{C_n} \frac{f(\zeta)}{(\zeta_1-z_1)(\zeta_2-z_2)\ldots (\zeta_n-z_n)}d\zeta_1 d\zeta_2 \ldots d\zeta_n\\
&=\frac{1}{2\pi i }^n\int_{C} \frac{f(\zeta)}{(\zeta_1-z_1)(\zeta_2-z_2)\ldots (\zeta_n-z_n)}d\zeta_1 d\zeta_2 \ldots d\zeta_n.
\end{align*}
For $\left|z_1\right|<\left|\zeta_1\right|$, we can write
\begin{align*}
\frac{1}{\zeta_1-z_1}&= \frac{1}{\zeta_1\left(1-\frac{z_1}{\zeta_1}\right)}\\
&=\zeta_1 \sum_{v_1}\frac{z_1^{v_1}}{\zeta_1^{v_1}}\\
&=\sum_{v_1}\frac{z_1^{v_1}}{\zeta_1^{v_1+1}}.
\end{align*}
Which implies that,
$$f(z_1,z_2,\ldots, z_n)=\frac{1}{(2\pi i)^n }\int_{C} \sum_{v_1} \frac{f(\zeta)z_1^{v_1}}{\zeta_1^{v_1+1}(\zeta_2-z_2)\ldots (\zeta_n-z_n)}d\zeta_1 d\zeta_2 \ldots d\zeta_n.$$\\
Proceeding in the same manner for $z_2$, we get
$$f(z_1,z_2,\ldots, z_n)=\frac{1}{(2\pi i)^n }\int_{C} \sum_{v_1v_2} \frac{f(\zeta)z_1^{v_1}z_2^{v_2}}{\zeta_1^{v_1+1}\zeta_2^{v_2+1}(\zeta_3-z_3)\ldots (\zeta_n-z_n)}d\zeta_1 d\zeta_2 \ldots d\zeta_n.$$
Continuing like this , we can write
$$f(z_1,z_2,\ldots, z_n)=\frac{1}{(2\pi i)^n }\int_{C} \sum_{v_1v_2\ldots v_n} \frac{f(\zeta)z_1^{v_1}z_2^{v_2}\ldots z_n^{v_n}}{\zeta_1^{v_1+1}\zeta_2^{v_2+1} \zeta_3^{v_3+1} \ldots \zeta_n^{v_n+1}}d\zeta_1 d\zeta_2 \ldots d\zeta_n.$$
Interchanging summation and integration, which is justified by Fubini's theorem,
$$f(z)=\sum_va_vz^v,$$
 where $\displaystyle{a_v=a_{v_1v_2\ldots v_n}=\frac{1}{(2\pi i)^n }\int_{C}  \frac{f(\zeta)}{\zeta_1^{v_1+1}\zeta_2^{v_2+1} \zeta_3^{v_3+1} \ldots \zeta_n^{v_n+1}}d\zeta_1 d\zeta_2 \ldots d\zeta_n}$. Hence $f$ is jointly analytic.

\end{proof}
\end{lem}
We now state Schwarz's lemma and its corollary without proof. For proof one can see ~\cite{conway2}.\begin{lem}$(\textbf{Schwarz's\ lemma})$
Let $g$ be a holomorphic functin on $\{z:|z|<1\}$, with $g(0)=0$ and $|g(z)|\leq 1$. Then $|g(z)|\leq |z|$ and $|g'(0)|\leq 1$.
\end{lem}
\begin{cor}
Let $g$ be a holomorphic function on $\{z:|z|<r\}$, with $|g(z)|\leq B$ for a bound $B$. Then for $z,\zeta$ in the disk, $$\left|g(z)-g(\zeta)\right|\leq 2B\left|\frac{r(z-\zeta)}{r^2-\bar{\zeta}z}\right|.$$
\end{cor}
\begin{lem}
Let $f$ be separately analytic and bounded on a closed polydisk $\displaystyle{D=\prod_{1\leq i\leq n}D_i}$, where $D_i=\{z_i\in{\mathbb{C}}:|z_i|\leq r_i\}$. Then $f$ is jointly continuous and hence analytic.
\begin{proof}
Let $\left|f(z)\right|\leq B$ for $z\in{D}$.
We will show that $$\left|f(z)-f(\zeta )\right|\leq 2B\sum_{1\leq j\leq n}\frac{r_j\left|z_j-\zeta_j\right|}{r_j^2- \overline{\zeta_j}z_j}.$$
Since 
\begin{align*}
f(z)-f(\zeta)=&f(z_1,z_2,\ldots, z_n)-f(\zeta_1,z_2,\ldots, z_n)\\
&   +f(\zeta_1,z_2,\ldots, z_n)-f(\zeta_1,\zeta_2,\ldots, z_n)+\ldots\\
&   +f(\zeta_1,\zeta_2,\ldots,\zeta_{n-1} ,z_n)-f(\zeta_1,\zeta_2,\ldots, \zeta_n).
\end{align*}
Therefore, it is enough to prove the inequality in one varibale case,
which follows from corollary to Schwarz's lemma.

\end{proof}
\end{lem}
We will prove Hartog's theorem by induction on $n$. Assume that Hartog's theorem is true for domains in $\mathbb{C}^{n-1}$. We prove it for domains in $\mathbb{C}^n$. For this we will need the following lemmas.
\begin{lem}
Let $f$ be separately analytic on a (non-empty) closed polydisk $\displaystyle{D=\prod_{1\leq i\leq n} D_i }$, where $D_i$ are closed disks in $\mathbb{C}$. Then there exists non-empty closed disks $E_i \subset D_i$, $1\leq i\leq n-1$, with $E_n=D_n$ such that $f$ is bounded on $\displaystyle{E=\prod_{1\leq i\leq n} E_i}$.
\begin{proof}
For each $r>0$, consider $$\Omega_r=\{z'\in{\prod_{1\leq i\leq n-1}D_i}:\leq\left|f(z',z_n)\right|\leq r \quad  \mbox{for all } \quad z_n\in{D_n}\}.$$
By induction hypothesis, for each fixed $z_n$, the map $z'\rightarrow f(z',z_n)$ is analytic and therefore continuous.
So, if a sequence $\{z_j'\} $ in $\Omega_r$ converges to $z'$, then for fixed $z_n\in D_n$ , since $f$ is continuous as a function of $n-1$ variables, we get that $$f(z_j',z_n) \rightarrow f(z',z).$$ 
Therefore we have,
\begin{align*}
|f(z',z_n)|&=|\lim f(z_j',z_n)|\\
&=\lim|f(z_j',z_n)|\\
&\leq \lim r\\
&=r,
\end{align*}
which gives that $ z'\in{\Omega_r}.$
So, we get that each of $\Omega_r$ is closed.\\
Also, it follows that $$\displaystyle{\cup_{r=1}^{\infty} \Omega_r} =\prod_{1\leq i\leq n-1}D_i.$$
Note that since each of $D_i$ is compact therefore $\prod_{1\leq i\leq n-1}D_i$ is compact and hence complete, which is expressed as a countable union of closed sets.\\
Therefore, by $Baire's \ category \ theorem$, which says that, a complete metric space can not be expressed as a countable union of nowhere dense subsets, we get that atleast one of $\Omega_r$ has non-empty interior. 
Let us assume that $\Omega_{r_0}$ has non-empty interior. Thus we can find $\displaystyle{\prod_{1\leq i\leq n-1} E_i \subset \Omega_{r_0}}$.
Also $f$ is bounded on $\displaystyle{E=\prod_{1\leq i\leq n} E_i}$, $E_n=D_n$.
\end{proof}
\end{lem}
Now we prove Hartog's theorem on separate analyticity, in the case when domain of the function is a polydisk. General case can be deduced from this.
We will also use Hartog's lemma on subharmonic functions, which we state here without proof:
\begin{lem}
Let $(u_\alpha)$ be a family of real-valued subhamronic functions in an open set $U$ of $\mathbb{C}$. Suppose that the functions are uniformly bounded from above, and that $$\limsup _{\alpha}u_{\alpha}(z)\leq C$$ for every $z\in{U}$. Then, given $\delta >0$ and compact $K\subset U$ there exists $\alpha_0$ such that for $z\in{K}$ and $\alpha \geq \alpha_0$ $$u_{\alpha}(z)\leq C+\delta.$$
\end{lem}
\begin{lem}
If for a power series $\displaystyle{\sum_{v}a_vz^v}$, the set $\{a_vw^v:v \ \mbox{ is a multi index}\}$, where $w\in{\mathbb{C}^n}$, is bounded. Then the power series $\displaystyle{\sum_{v}a_vz^v}$ is convergent in polydisk $P(0,w)$.
\begin{proof}
Let B be the bound of the set $\{a_vw^v:v \ \mbox{is a multi index}\}$. And let $z\in{P(0,w)}$, then $\left|z_i\right|<w_i$. That is $\left|\frac{z_i}{w_i}\right|<1$, and therefore the series $\displaystyle{\sum_v\left|\frac{z}{w}\right|^v}$ is convergent. Therefore there exists $k\in{\mathbb{N}}$ such that $$\sum_{|v|\geq k}\left|\frac{z}{w}\right|^v<\frac{\epsilon}{B}.$$
For $|v|\geq k$, consider 
\begin{align*}
\sum_{|v|\geq k}a_vz^v|&\leq \sum_{|v|\geq k}|a_vz^v|\\
&=\sum_{|v|\geq k}|a_vw^v|\left|\frac{z^v}{w^v}\right|\\
&\leq \sum_{|v|\geq k} B \left|\frac{z^v}{w^v}\right|\\
&<\epsilon.
\end{align*}
Hence the series $\displaystyle{\sum_{v}a_vz^v}$ is convergent in $P(0,w)$, since $z\in{P(0,w)}$ was arbitrary.
\end{proof}
\end{lem}
The above mentioned lemma is known as Abel's lemma.
\begin{lem}
Let $f$ be a separately analytic function on a polydisk $\displaystyle{D=\prod_{1\leq i\leq n} D_i}$, where $D_i=\{z_i \in{\mathbb{C}}:\left|z_i\right|<r\}$ are disks in $\mathbb{C}$. Then $f$ is jointly analytic in $D$.
\begin{proof}
By lemma $2.1.3$, we can get a smaller poldisk, say $E\subseteq D$, $\displaystyle{E=\prod_{1 \leq i\leq n-1 }} E_i \times D_n$ such that $f$ is bounded in $E$ with bound, say, $B$, where radius of disk $E_i$ is $\epsilon \leq r$ for $i=1,2,\ldots, n-1$.\\
Since $f$ is analytic as a function of $n-1$ variables, therefore for fixed $z_n$, we have $$f(z)=c_\alpha (z_n)(z')^\alpha \quad \mbox{ for }\quad z'\in{\prod_{1 \leq i\leq n-1 } E_i},$$
where $\displaystyle{c_\alpha (z_n)=\frac{\partial ^{\alpha }}{\partial z'^\alpha } \frac{f(0,z_n)}{\alpha !}}$.
By Cauchy's integral formula in $z'$, \\
$\frac{\partial ^{\alpha }}{\partial z'^\alpha } \frac{f(0,z_n)}{\alpha !}=\frac{1}{(2\pi i )^{n-1}} \int_{C_1}\int_{C_2}\ldots \int_{C_{n-1}} \frac{f(\zeta)}{(\zeta_1-z_1)^{\alpha_1+1}(\zeta_2-z_2)^{\alpha_2+1}\ldots (\zeta_{n-1}-z_{n-1})^{\alpha_{n-1}+1}}d\zeta_1d\zeta_2\ldots d\zeta_{n-1}.$
Proceeding as in Osgood's lemma, again by expanding geometric series and interchanging summation and integration, it can be shown that $c_\alpha (z_n)$ are analytic functions as a function of variable $z_n$.\\
For $0<r_1<r_2<r$ and fixed $z_n$ with $\left|z_n\right|<r$, from the convergence of the power series, we get \begin{equation} \left|c_\alpha (z_n)\right| {r_2}^{\left|\alpha\right|} \rightarrow 0 \quad \mbox{as} \quad \left|\alpha \right|\to \infty.
\end{equation}
Also, since the map $z_n\to c_\alpha (z_n)$ is analytic, therefore $\log |c_\alpha (z_n)|$ is harmonic and therefore subharmonic.
Consider family $(u_{\alpha})$ of subharmonic functions, where $u_{\alpha}(z_n)=\frac{1}{\left|\alpha\right|}\log|c_\alpha (z_n)|.$
From Cauchy's inequality, we have $$\left|c_\alpha (z_n)\right|\epsilon^{\left|\alpha\right|} <B,$$\\
which gives $\log \left|c_\alpha (z_n)\right| < \log B- \left|\alpha\right| \log\epsilon $.
Therefore we have $u_{\alpha}(z_n)<\left[\frac{\log B}{\left|\alpha\right|}-\log \epsilon\right]$. And thus
the family $(u_{\alpha})$ is uniformly bounded above.
Also from (2.1), we get that 
\begin{equation}
\left|u_{\alpha}(r_2)\right|<\log\left(\frac{1}{r_2}\right)\ \mbox{ as }\ \left|\alpha\right| \to \infty.
\end{equation}
For $r_1<r_2$, $\log r_1<\log r_2$.
Thus 
$\log \left(\frac{1}{r_2}\right) <\log \left(\frac{1}{r_1}\right)$.
Let $\delta_0$ be such that $\log \left(\frac{1}{r_2}\right)+\delta_0 \leq \log \left(\frac{1}{r_1}\right)$.
By, Hartog's lemma on subharmonic functions for $\delta_0$ and compact set $\{z_n:|z_n|<r_1\}$, we get that for large $|\alpha|$ $$u_{\alpha}(z_n)\leq \log \left(\frac{1}{r_2}\right)+\delta_0.$$
And $\log \left(\frac{1}{r_2}\right)+\delta_0 \leq \log \left(\frac{1}{r_1}\right)$ gives that $$u_{\alpha}(z_n)\leq \log \left(\frac{1}{r_1}\right).$$
Thus for large $\left|\alpha\right|$ $$\left|c_{\alpha}(z_n)\right|r_1^{\left|\alpha\right|}\leq 1$$
uniformly on $|z_n|<r_1$. By Abel's lemma, the series $f(z',z_n)=\displaystyle{\sum_{\alpha}c_{\alpha}(z_n)z'^{\alpha}}$ converges absolutely and uniformly on any compact subset of $D$. Hence $f$ is analytic in the polydisk $D$, since eacch term in the series is analytic.
\end{proof}
\end{lem}
Now we finally present proof of Hartog's result, which follows easily from the several lemmas stated above.
\begin{thm}$(\textbf{\mbox{Hartog's theorem on separate analyticity}})$ Let $f$ be a complex valued function defined on an open subset $U$ of $\mathbb{C}^n$. Suppose that $f$ is separately analytic, then $f$ is analytic as a function of $n$ variables.
\begin{proof}
Let $z\in{U}$, choose $r>0$ such that the polydisk of radius $2r$ around z is contained in $U$. Then by lemma $2.1.3 $, we get that there is a polydisk $D$ containing z and a smaller polydisk $E$ inside $D$ such that $f$ is analytic on $E$. Following the argument of lemma $2.1.6 $, we see that the power series for $f$ on smaller polydisk converges on the larger polydisk $D$. Since $D$ contains the point $z$, we get that $f$ is analytic at $z$. Hence $f$ is analytic on $U$. 
\end{proof}
\end{thm}
\section{Relation with theory of one variable}
Let $n>1$, consider $\mathbb{C}^n$ consisting of $n$-tuples of complex variables. One of the most important results of theory of one complex variable is Cauchy's integral formula. Several important results are followed from this inequality. We will now prove this integral formula for functions analytic on domains in $\mathbb{C}^n$.
\begin{thm}
$\textbf{\mbox{(Cauchy's Integral Formula)}}$ If $f$ is analytic in a domain $D$ of $\mathbb{C}^n$. Let $a=(a_1,\ldots, a_n)$ be a point in $D$. Let $$ \overline{\triangle}=\{(z_1,\ldots, z_n):|z_i-a_j|\leq r_j\}$$ be a closed polydisk centered at $a$ which is contained in $D$. Then for $z\in{\triangle}$
$$f(z)=\left(\frac{1}{2\pi i}\right)^n\int_{\partial \triangle} \frac{f(\zeta_1,\ldots, \zeta_n)}{(\zeta_1-z_1)\ldots (\zeta_n-z_n)}d\zeta_1\ldots d\zeta_n.$$
\begin{proof}
We will prove it by induction on $n$. For $n=1$, Cauchy integral formula holds for functions of one complex variable.
Assume the result for functions defined on domains of $\mathbb{C}^{n-1}$.\\
Let $(z_1,\ldots, z_n)\in{\triangle}$ be fixed. Since $f$ is analytic as a function of variable $z_1$, we can write $$f(z_1,z_2,\ldots, z_n)=\frac{1}{2\pi i}\int_{C_1}\frac{f(\zeta_1,z_2,\ldots, z_n)}{\zeta_1-z_1}d\zeta_1,$$
where $C_1$ is a circle centered at $a_1$ bounding $z_1$.\\
Now for any fixed point $\zeta_1$ on the circle $C_1$, $f(\zeta_1,z_2,\ldots, z_n)$ is an analytic function of $n-1$ variables $z_2,\ldots, z_n$. Therefore it follows from inductive hypothesis that $$f(\zeta_1,z_2,\ldots, z_n)=\frac{1}{(2\pi i)^{n-1}}\int_{C_2} \ldots \int_{C_n}\frac{f(\zeta_1,\zeta_2,\ldots, \zeta_n)}{(\zeta_2-z_2)\ldots (\zeta_n-z_n)}d\zeta_2\ldots d\zeta_n$$ for $z_i\in{\triangle_0}=\{(z_2,\ldots , z_n):|z_j-a_j| \leq r_j, j=2,\ldots ,n \}$
where $C_i: 2\leq i \leq n$ is a circle centered at $a_i$ bounding $z_i$.\\
On substitution we get $$f(z_1,z_2,\ldots, z_n)=\frac{1}{(2\pi i)^{n}}\int_{C_1} \ldots \int_{C_n}\frac{f(\zeta_1,\zeta_2,\ldots, \zeta_n)}{(\zeta_1-z_1)(\zeta_2-z_2)\ldots (\zeta_n-z_n)}d\zeta_1\ldots d\zeta_n.$$
\end{proof}
\end{thm}
It follows from Cauchy's integral formula that any analytic function $f$ in $D$ has partial derivatives of all order with respect to each variables $z_j$ (j=1,...,n). From the proof of Osgood's lemma we note that any analytic function on $D\subseteq \mathbb{C}$ can be written in the form $f(z)=\sum_va_vz^v$;\\
 where $$a_v=a_{v_1v_2\ldots v_n}=\frac{1}{(2\pi i)^n }\int_{C}  \frac{f(\zeta)}{\zeta_1^{v_1+1}\zeta_2^{v_2+1} \zeta_3^{v_3+1} \ldots \zeta_n^{v_n+1}}d\zeta_1 d\zeta_2 \ldots d\zeta_n
.$$
From Cauchy's integral formula we can write any partial derivative $\displaystyle{\frac{\partial^{v_1+\ldots +v_n}}{\partial z_1^{v_1}\ldots \partial z_n^{v_n}}}$ of $f$ as $$\frac{v_1!\ldots v_n !}{(2\pi i)^n }\int_{C}  \frac{f(\zeta)}{(\zeta_1-z_1)^{v_1+1}\ldots (\zeta_n-z_n)^{v_n+1}}d\zeta_1 d\zeta_2 \ldots d\zeta_n.$$
This gives that $$a_v(z)=\frac{1}{v_1!\ldots v_n!}\frac{\partial ^{v}f(0)}{\partial z_1^{v_1}\ldots \partial z_n^{v_n}},$$
and which in turn gives analogus of Cauchy's representation formula for theory of one complex variable.
\begin{rem}
Another proof for cauchy's integral formula follows from proof of Osgood's lemma.
\end{rem}
\begin{thm} $(\mbox {\textbf{Identity Theorem}})$
Let $f$ and $g$ be analytic in an open connected subset $D$ of $\mathbb{C}^n$. If $f(z)=g(z)$ for all $z\in{U}$; where $U$ is a non-empty open subset of $D$. Then $f(z)\equiv g(z)$ in $D$.
\begin{proof}
Let us denote the interior of the set $\{z\in{D}:f(z)=g(z)\}$ by $E$. Therefore $E$ is open. If we show that $E$ is relatively closed, then using connectivity of $D$ we can say that $E=D$ and then we are done.\\
Let us denote by $ \overline{E}$, closure of $E$ in $\mathbb{C}$ and let $w\in{D\cap  \overline{E}}$. Since $D$ is open, we can choose a number $r>0$ such that the polydisk, $P(w,r)$ around $w$ of radius $r$ is contained in $D$. Also since $w\in{ \overline{E}}$, therefore there is $w'\in{E}$ such that $w'\in{\triangle \left(w,\frac{r}{2}\right)}$; where $\triangle \left(w,\frac{r}{2}\right)$ is a polydisk centered at $w$ of radius $\frac{r}{2}$.\\
It is clear that $f(z)-g(z)$ is analytic in $\triangle \left(w',\frac{r}{2}\right)$ and hence has a power series representation around point $w'$ having radius of convergence greater than or equal to $\displaystyle{\frac{r}{2}}$.\\ Since $w'\in{E}$, $f(z)-g(z)=0$ in an open neighbourhood of $w'$, therefore all the coefficients in the power series expansion of $f$ vanish. Thus $$f(z)-g(z)\equiv 0$$ on $\triangle (w',r)$ and noting that $w\in{\triangle (w',r)}$, we get that $w\in{E}$. Which shows that $E$ is relatively closed in $D$. Thus $E$ is an open and closed subset of $D$. By using connectivity of $D$, we get that $D=E$. Hence $f(z)\equiv g(z)$ on $D$.
\end{proof}
\end{thm}
\begin{rem}
In theory of one complex variable, we study that zero set of an analytic funtion contains no limit point. In other words, zero set of an analytic function of one complex variable contains isolated points. But this is far from being true in function theory of several complex variables.\\
\end{rem}
\textbf{Example}: Consider function $f:\mathbb{C}^2\to \mathbb{C}$ defined as $$f(z,w)=z\quad \mbox{for} \quad z,w \in {\mathbb{C}}.$$
It is easy to see that $f$ is analytic on $\mathbb{C}^2$.
Here zero set of of $f$ is given by $$\{(0,w):w\in{\mathbb{C}}\}$$ contains no isolated points. As any neighbourhood of a zero point $(0,w_0)$ will definitely contain points of the form $(0,w):w\neq w_0$.
Contrary to the case of one complex variable, zero set of an analytic function in a domain $D\subseteq \mathbb{C}^n,n\geq 2$, contains no isolated points.

Identity theorem says that a the zero set of a non-zero analytic map on an open connected set can not have non-empty interior. This is also true in the theory of one complex variable. In one variable complex theory, something stronger is true. Which is that zeroes of a non-zero analytic function are isolated. This does not hold in domains in $\mathbb{C}^n;n\geq 2$, it can be seen by the previous example. More generally, consider function $f:\mathbb{C}^n\to \mathbb{C}$ defined as $$f(z_1,\ldots,z_n)=z_1.$$
Zeroes of $f$ is given by the set $\{(0,z_2,\ldots,z_n):z_i\in{\mathbb{C}, i=2,\ldots,n}\}$; which does not contain isolated points. In fact, in general we have the result that a function analytic in a domain of $\mathbb{C}^n;n\geq 2$ has no isolated zeroes. We will prove this result by using Hurwitz's theorem. We first state Hurwitz's theorem and then using it we will prove that an analytic function $\mathbb{C}^n;n\geq 2$ can not have isolated zeroes. 
\begin{thm}
$\mbox{(\textbf{Hurwitz's theorem})}$ Let $G\subseteq \mathbb{C}$ be open and connected. Let $f_n$ be a sequence of analytic functions on $G$ such that $f_n\to f$ uniformly on compact subsets of $G$, where $f$ is analytic on $G$. Assume that $f$ is not identically zero on a closed ball $\overline{B(a,r)}\subseteq G$ and does not vanish on $|z-a|=r$. Then there exists an $N\in{\mathbb{N}}$ such that for all $n\geq N$, $f_n$ and $f$ have the same number of zeroes in $B(a,r)$.   
\end{thm}
\begin{rem}
This result does not hold if $f$ has zeroes on boundary of the disk. For example consider a sequence of	 analytic functions ($f_n$), defined on open unit disk $D\subseteq \mathbb{C}$ given by $$f(z)=z-1+\frac{1}{n};\ z\in{\mathbb{C}}.$$
$f_n$ converges to an analytic function $f$ given by $$f(z)=z-1.$$ $f$ has no zeroes in the disk $D$, but each $f_n$ has a zero in $D$, namely, $1-\frac{1}{n}$. Here Hurwitz's theorem does not hold since $f$ vanishes at a point of the unit circle $\partial D$.
\end{rem}
\begin{thm}
Let $\Omega \subseteq \mathbb{C}^n,n\geq 2$ be a region. Then $f$ has no isolated zeroes.
\begin{proof}
Assume that $f$ has a zero at $z_0\in{\Omega}$, that is $f(z_0)=f(z_1,\ldots,z_n)=0$. Consider a sequence of functions $g_k$ defined on $\Omega_1$ as $$g_k(z)=f\left(z,z_2+\frac{1}{k}+\ldots +z_n+\frac{1}{k}\right),.$$ where $\Omega_1=\{z_1:(z_1,z_2,\ldots, z_n)\in{\Omega}\}.$
Then clearly $g_k$ is a sequence of analytic functions which converges to an analytic function $g$ given by $$g(z)=f(z,z_2,\ldots,z_n).$$ We see that $g(z_1)=0$, let $B(x_0,r)$ be a neighborhood of $z_0$ such that $g\neq 0$ on $\big(B(x_0,r)\setminus \{x_0\}\big) \cup \partial B$ (we can choose such a neighborhood since zeroes of an analytic functions in domains of $\mathbb{C}$ are isolated) . Therefore by Hurwitz's theorem, there is an integer $K \in{\mathbb{N}}$ such that $g_k$ and $g$ have same number of zeroes in $B(x_0,r)$, therefore we get $$g_k\left(z_1^k\right)=0\ \mbox{for} \ z_1^k\in{\mathbb{C}};\  \left|z_1-z_1^k\right|< r .$$ That is $$f\big(z_1^k,z_2+\frac{1}{k}+\ldots +z_n+\frac{1}{k}\big)=0.$$ If $\displaystyle{r\leq \frac{n}{\epsilon}}$, then for $k\geq K$ with $\displaystyle{k>\frac{n}{\epsilon}}$, we see that 
\begin{align*}
||\left(z_1^k,z_2+\frac{1}{k}+\ldots +z_n+\frac{1}{k}\right)-\left(z_1,z_2,\ldots,z_n\right)||&= \left|z_1^k-z_1\right|+\left|z_2+\frac{1}{k}-z_2\right|+\ldots +\left|z_n+\frac{1}{k}-z_n\right|\\
&<\frac{\epsilon}{n}+\frac{\epsilon}{n}\ldots +\frac{\epsilon}{n}\\
&=\epsilon.
\end{align*}
In case $\displaystyle{\frac{\epsilon}{n}<r}$, then applying Hurwitz's theorem on $\displaystyle{B'\left(z_0,\frac{\epsilon}{n}\right)}$ and following the same procedure as above, we can get the same inequality as above.
Thus we see that $f$ has zeroes arbitrarily close to $(z_1,z_2,\ldots,z_n)$ and since $(z_1,z_2,\ldots,z_n)$
 was an arbitrary zero, we get that zeroes of $f$ are not isolated.
\end{proof}
\end{thm}

\chapter{Invertibility and Jacobian}
 
In this chapter, notion of Jacobian of a function is introduced. We will discuss relation of Jacobian of a function with the injectivity of the function. We will discuss inverse function theorem for functions analytic on domains of $\mathbb{C}$. Then we will discuss Lewy's theorem for harmonic mappings. We will also discuss the generalization of Lewy's theorem for pluriharmonic functions given by M. Naser. (See ~\cite{lewy3} and ~\cite{naser3}).
\section{One-dimensional case}
 We begin this section by giving definition and some properties of Jacobian of a function defined on $\mathbb{C}^n$.\\
\textbf{\underline{Jacobian of functions on $\mathbb{C}^n$}}\\
Let $\Omega \subseteq \mathbb{C}^n$ be a region
and $f:\Omega \to \mathbb{C}^n$ be any differentiable map,
then $f$ can be written as $(f_1,f_2,\ldots ,f_n)$;
where $f_i:\Omega \rightarrow \mathbb{C}$, $(1\leq i \leq n)$ is given by
$$f(z)=(f_1(z),f_2(z),\ldots,f_n(z)).$$
Then for any $z\in {\Omega}$
$$f(z+h)=f(z)+hf'(z)+o(|h|^2),$$ for h near origin in $\mathbb{C}^n$.
For $1\leq k \leq n$,
let $h=\lambda e_k=(0,0,\ldots ,\lambda ,0\ldots ,0)$; where $\lambda $ is an arbitrary scalar. Then 
\begin{align*}
f(z+\lambda e_k)&=f(z)+\lambda e_k f'(z)+o\left(\left|\lambda e_k\right|^2\right)\\
&=f(z)+\lambda e_k f'(z)+o(\left|\lambda \right|^2).\  \mbox{(since} \ |e_k|=1 \  \mbox{for \ all} \  k)
\end{align*}
Thus $ \frac{f\left(z+\left(0,0,\ldots,\lambda,0,\ldots ,0 \right)\right)-f(z)}{\lambda}=f'(z)e_k+o\left(\frac{\left|\lambda\right|^2}{\lambda}\right)$.
Taking limits as $\lambda \rightarrow 0$, we get-
$$\left(D_k f\right)(z)=f'(z)e_k$$
or 
\begin{align*}
f'(z)e_k&=\left(D_k f\right)(z)\\
&=\sum_{j=1}^{n}(D_kf_j)(z)e_j.
\end{align*}
Which gives $k$-th column vector in the matrix representation of map $f'$ with respect to the standard basis of $\mathbb{C}^n$.
Matrix representation of $f'$, called Jacobian matrix of $f$ is given by

$$\begin{bmatrix}
\frac{\partial f_1}{\partial z_1} & \frac{\partial f_1}{\partial z_2}&\ldots & \frac{\partial f_1}{\partial z_n}\\
\frac{\partial f_2}{\partial z_1} & \frac{\partial f_2}{\partial z_2}&\ldots & \frac{\partial f_2}{\partial z_n}\\
\ldots & \ldots & \ldots\\
\frac{\partial f_n}{\partial z_1} & \frac{\partial f_n}{\partial z_2}&\ldots & \frac{\partial f_n}{\partial z_n}\\
\end{bmatrix}.$$
\\
Determinant of this matrix, denoted as $J_{f}(z)$ is called Jacobian of $f$. Note that since determinant is a continuous function, therefore if $J_f(z)\neq 0$, then either $J_f> 0$ or $J_f(z)<0$. If $J_f> 0$ we say $f$ sense preserving or orientation preserving and if $J_f(z)<0$, then $f$ is said to be sense reversing or orientation reversing.\\
\underline{\textbf{Note}}:
Denoting variables $z_i$ as $x_i+y_i$ and $f_i$ as $u_i+v_i$,
consider $f$ as a map $$f:\mathbb{R}^{2n}\to \mathbb{R}^{2n} \ \mbox{defined as}$$ 
$$f(w)=(u_1(w),v_1(w),\ldots,u_n(w),v_n(w)),$$ where $w=(x_1,y_1,\ldots,x_n,y_n).$
Then Jacobian matrix of $f$ is given by\\
$$\begin{bmatrix}
\frac{\partial u_1}{\partial x_1} & \frac{\partial u_1}{\partial y_1}&\ldots &\frac{\partial u_1}{\partial x_n}& \frac{\partial u_1}{\partial y_n}\\
\frac{\partial v_1}{\partial x_1} & \frac{\partial v_1}{\partial y_1}&\ldots &\frac{\partial v_1}{\partial x_n}& \frac{\partial v_1}{\partial y_n}\\
\ldots & \ldots & \ldots&\ldots &\ldots\\
\frac{\partial v_n}{\partial x_1} & \frac{\partial v_n}{\partial y_1}&\ldots &\frac{\partial v_n}{\partial x_n}& \frac{\partial v_n}{\partial y_n}\\
\end{bmatrix}.$$
Jacobian of this matrix, denoted as $J_{\mathbb{R}(f)}$ is called real Jacobian of $f$.
By applying appropriate row or column transformations, $J_{\mathbb{R}(f)}$ is same as
$$\mbox{det}\begin{bmatrix}
\left(\displaystyle{\frac{\partial u_i}{\partial x_j}}\right) & \left(\displaystyle{\frac{\partial u_i}{\partial y_j}}\right)\\
 \ldots & \ldots\\
\left(\displaystyle{\frac{\partial v_i}{\partial x_j}}\right) & \left(\displaystyle{\frac{\partial v_i}{\partial y_j}}\right)\\
\end{bmatrix},$$
where the blocks in the matrix are $n\times n$ matrices with real entries. So if we assume that $f$ is analytic, then it is analytic in each variable separately.
Therefore by Cauchy Riemann equations for each $1\leq j\leq n$, we get that
$$\frac{\partial u_i}{\partial x_j}=\frac{\partial v_i}{\partial y_j}\ \mbox{ and}
 \ \frac{\partial u_i}{\partial y_j}= - \frac{\partial v_i}{\partial x_j}.$$
Now again using appropriate row or column transformations involving $ i$, and using these Cauchy Riemann equations observe that $J_{\mathbb{R}(f)}$ is same as
$$det\begin{bmatrix}
\left(\displaystyle{\frac{\partial u_i}{\partial x_j}+ i \frac{\partial v_i}{\partial x_j}}\right) & \left(\displaystyle{\frac{\partial u_i}{\partial x_j}- i \frac{\partial v_i}{\partial x_j}}\right )\\
 \ldots & \ldots\\
\left(\displaystyle{\frac{\partial v_i}{\partial x_j}}\right) & \left(\displaystyle{\frac{\partial v_i}{\partial x_j}}\right)\\
\end{bmatrix}.$$
Now Subtracting columns on the right block from $ i $ times the left blocks, we get that
$$J_{\mathbb{R}(f)} =
 det\begin{bmatrix}
\left(\displaystyle{\frac{\partial f_i}{\partial x_j}}\right)& 0\\
 \ldots & \ldots\\
\left(\displaystyle{\frac{\partial v_i}{\partial x_j}}\right) & \left(\displaystyle{\frac{\partial  \overline{f_i}}{\partial x_j}}\right)\\
\end{bmatrix}.$$
So we get that $ J_{\mathbb{R}}(f)=|f(z)|^2$.\\
\begin{thm} Let $f:\Omega \subseteq \mathbb{C} \rightarrow \mathbb{C}$ be analytic. Then $f$ is locally one-one at $z_0$ if and only if the Jacobian $J_f(z_0)=f'(z_0)\neq 0$.
\begin{proof}
Let $f$ be analytic and locally one-one at $z_0\in{\Omega}$, let $V\subseteq \Omega$ be such that $f\lvert_V$ is one-one. Denote this restriction by $g$. Let us assume that $g'(z_0)=0$. Thus $z=z_0$ is a zero of $g(z)-g(z_0)$ of order $k$, $k>1$. Therefore by theorem $1.1.7$ by sufficiently small $\epsilon>0$ there exists $\delta >0$ such that for $\zeta$ with $\displaystyle{\left|g(z_0)-\zeta\right|<\delta}$, the equation $$g(z)=\zeta$$ has exactly $k$ simple roots in the disk $|z-z_0|<\epsilon$. Which gives that $g$ can not be one-one, which is a contradiction. Thus $g'(z_0)=f'(z_0)\neq 0$.\\
Conversely, assume that $f'(z_0)\neq 0$, $z_0\in{\Omega}$. Thus the function  $f(z)-f(z_0)$ has a zero of order $1$ at $z=z_0$. And therefore again by theorem $1.1.7$, for sufficiently small $\epsilon >0$ there is a $\delta>0$ such that for $\zeta$ with $\left|f(z_0)-\zeta\right|<\delta$ (denote this ball by $D$), the equation $$f(z)=\zeta$$ has exactly one root in the disk $B=\{z:|z-z_0|< \epsilon \}$. Take $U=f^{-1}(D)\cap B$, it is easy to check that $f$ is one-one on $U$. Hence we obtain that $f$ is locally one-one at $z_0$.

\end{proof}
\end{thm}
Now, we will discuss the above result for harmonic maps.
$\mbox{That is},$ for a one-one harmonic map on $\mathbb{C}$, Jacobian does not vanish. We will need the following definitions.
\begin{defi}
For a map $u:U\subseteq \mathbb{R}^n\rightarrow \mathbb{R}$, gradient of $u$, denoted as $\mathsf{grad}\  u$, is defined as $\displaystyle{\left(\frac{\partial u}{\partial x_1},\frac{\partial u}{\partial x_2},\ldots , \frac{\partial u}{\partial x_n}\right)}$.
\end{defi}
\begin{rem} Let $u:U\subseteq \mathbb{R}^2\rightarrow \mathbb{R}$ be any map. Then Taylor expansion of $u$ about a point $(a_1,a_2,\ldots , a_n)\in {\mathbb{R}^n}$ is given by \\
$u(x_1,x_2,\ldots , x_n)=u(a_1,a_2,\ldots , a_n)+(x_1-a_1)u_{x_1}(a_1,a_2,\ldots , a_n)+(x_2-a_2)u_{x_2}(a_1,a_2,\ldots , a_n)\\+\ldots +(x_n-a_n)u_{x_n}(a_1,a_2,\ldots , a_n)+\frac{1}{2!}\left[\sum_{i,j=1}^{n}(x_i-a_i)(x_j-y_j)u_{x_ix_j}(a_1,a_2,\ldots , a_n)\right]+\ldots $\\
In particular, Taylor expansion of $u$ around origin $(0,0)\in {\mathbb{R}^2}$ is given by
\begin{align*}
u(x,y)&=u(0,0)+xu_x(0,0)+yu_y(0,0)+\frac{1}{2!}\bigg[(x-a)^2u_{xx}(0,0)\\
&\ \ \ \ \ \ +(x-a)(y-b)u_{xy}(0,0)+(y-b)(x-a)u_{yx}(0,0)+(y-b)^2u_{yy}(0,0)\bigg]+\ldots.
\end{align*}
\end{rem}
\begin{thm}
Let the mapping $u=(u_1,u_2):U\subseteq \mathbb{R}^2\rightarrow \mathbb{R}^2$ be one-one and continuous in a neighbourhood $U$ of origin $(0,0)\in {\mathbb{R}^2}$ and let $u(0,0)=(0,0)$. If the function $u_1$ is harmonic, then  $\mathsf{grad}\ u_1(0,0) \neq 0$.
\begin{proof}
Let us assume that $\mathsf{grad}\ u_1(0,0) = 0$.\\
$\mbox{That is} \ \displaystyle{\left(\frac{\partial u_1}{\partial x}(0,0),\frac{\partial u_1}{\partial y}(0,0)\right)=0}$,
so the gradient of $u_1$ takes form
\begin{align*}
u_1(x,y)&=u_1(0,0)+\frac{1}{2!}\bigg[(x-a)^2u_{1xx}(0,0)+(x-a)(y-b)u_{1xy}(0,0)\\
&\ \ \ \ \ \ \ +(y-b)(x-a)u_{1yx}(0,0)+(y-b)^2u_{1yy}(0,0)\bigg]+\ldots\\
&=\frac{1}{2!}\bigg[(x-a)^2u_{1xx}(0,0)+(x-a)(y-b)u_{1xy}(0,0)\\
&\ \ \ \ \ \ \ \  +(y-b)(x-a)u_{1yx}+(y-b)^2u_{1yy}(0,0)\bigg]+\ldots
\end{align*}
(since $u(0,0)=(0,0) \ \mbox{implies that} \ u_1(0,0)=0$)\\
So, we get that order of zero of $u_1$ at $(0,0)$ is $\geq 2$, say k. Now, since $u_1$ is harmonic, therefore $u_1$ locally has a harmonic conjugate, say $v_1$ such that $u_1+ i v_1$ is analytic, denote this function by $f$.
So, $u_1=\Re f$, where $f$ is a function analytic on $U$.
Consider the set $M=\{x\in U:u_1(x)=0\}$.
$M$ consists of $k$ smooth arcs, $\displaystyle{\gamma_{i},\ \left(i=1,\ldots,k\right)}$ which intersect at 0. To see this note that $u_1$ can locally be written as real part of an analytic function $f$ and $\mathsf{grad}\  u_1(0)=0$ implies that $f'(0)=0$, which gives that $z=0$ is a zero of order $k\geq 1$. Thus $f$ can be written as $f(z)=z^kg(z)$, $g(z)\neq 0$, which gives that $M$ consists of $k$ number of smooth arcs. Since for $x\in{M}$,  $u_1(x)=0$ and $u(0,0)=(0,0)$, we notice that the image of each of $\gamma_i$ under the map $u_2$ is an interval on the $u_2$ axis containing the origin (image of each such arc under $u_2$ consists origin, since $(0,0)\in{M}$ and $u(0,0)=(0,0)$). Thus we get that origin separates the set $M$ into $2k$ non intersecting parts in the plane $x=(x_1,x_2)$, and the image of orgin in the $u$ plane separates the image of $M$ into two parts. And since $2k>2$, we get a contradiction. Hence we get that $\mathsf{grad} \ u_1(0,0)\neq 0$. 
 
\end{proof}
\end{thm}
\begin{cor}$\mbox{\textbf{(Theorem of Lewy)}}$
Let the mapping $u=(u_1,u_2):U\subseteq \mathbb{R}^2\rightarrow \mathbb{R}^2$ be one-one and harmonic in a neighborhood $U$ of origin $(0,0)\in {\mathbb{R}^2}$. Then the Jacobian $\displaystyle{J(x,y)=\left[\frac{\partial (u_1,u_2)}{\partial (x,y)}\right]}$ does not vanish at the origin.
\end{cor}
\begin{proof}
If $u(0,0)=c$, where $c\in {\mathbb{R}^2}$ is non-zero, then by considering map $u'$ given by 
$u'(x,y)=u(x,y)-c$, we get $u'$ is one-one, harmonic with $u'(0,0)=(0,0)$. So without loss of generality we may assume that $u(0,0)=(0,0)$.
So the taylor expansion for $u_1$ about origin in $\mathbb{R}^2$ takes form
$$\displaystyle{u_1(x,y)=xu_{1x}(0,0)+yu_{1y}(0,0)+\ldots}.$$
Writing $a_1$ for $u_{1x}(0,0)$ and $a_2$ for $u_{1y}(0,0)$, we get this representation as 
$$u_1(x,y)=a_1x+a_2y+\ldots.$$
Similarly taylor expansion for $u_2$ about origin in $\mathbb{R}^2$ becomes
$$u_2(x,y)=b_1x+b_2y+\ldots,$$ where $b_1=u_{2x}(0,0)$ and $b_2=u_{2y}(0,0)$
So Jacobian matrix of $u$ is given by
$$\begin{bmatrix}
\displaystyle{\frac{\partial u_1}{\partial x}}&\displaystyle{\frac{\partial u_1}{\partial y}}\\
\displaystyle{\frac{\partial u_2}{\partial x}}&\displaystyle{\frac{\partial u_1}{\partial y}}
\end{bmatrix}\\
=\begin{bmatrix}
a_1+\ldots&a_2+\ldots\\
b_1+\ldots&b_2+\ldots
\end{bmatrix},$$where dots denote terms containing x and y's. So Jacobian of $u$ takes the form
$$J(x)=\left[\frac{\partial (u_1,u_2)}{\partial (x,y)}\right]\\
=a_1b_2-a_2b_1+\ldots,$$ where dots denote the terms containing x and y's.\\
Therefore $$J(0)=a_1b_2-a_2b_1.$$
Assume that the Jacobian of $u$ vanishes at the origin. Therefore we get that
$$a_1b_2-a_2b_1=0.$$
$\mbox{That is} \  \displaystyle{\frac{a_1}{b_1}=\frac{a_2}{b_2}}$, call this ratio as $\lambda$,
then $a_1=\lambda b_1$ and $a_2=\lambda b_2$.
In this case, consider map $v$ defined on $U$ as 
$$v=(v_1,v_2),\ \mbox{ where} \ v_1=u_1-\lambda u_2 \ \mbox{ and} \ v_2=u_2.$$
Consider gradient of $v_1$
\begin{align*}
\mathsf{grad}\ v_1&=\left(\frac{\partial v_1}{\partial x},\frac{\partial v_1}{\partial y}\right)\\
&=\left(\frac{\partial u_1}{\partial x}-\lambda \frac{\partial u_2}{\partial x},\frac{\partial u_1}{\partial y}-\lambda \frac{\partial u_2}{\partial y}\right)\\
&=\big(a_1-\lambda b_1+\ldots,a_2-\lambda b_2+\ldots\big),
\end{align*}
 where dots denotes terms containing x and y's.
Thus we get that $\mathsf{grad}\ v_1$ vanishes at the origin.
Note that $v$ is harmonic, since $u_1$ and $u_2$ are both harmonic.
And since the map $u\rightarrow v$ is one-one and also the map $x\rightarrow u$ is one-one.
Therefore, the map $x\rightarrow v$ is one-one, since composition of one-one maps is one-one.
Also $g_1$ is harmonic, so we get that $\mathsf{grad}\ v_ 1\neq0$ (by theorem $3.1.2$).
Hence we get a contradiction to our assumption that $J(0)=0$.
\end{proof}
\begin{cor}
Let the mapping $u=(u_1,u_2):U\subseteq \mathbb{R}^2\rightarrow \mathbb{R}^2$ be one-one and pluriharmonic in a neighbourhood $U$ of origin $(0,0)\in {\mathbb{R}^2}$. Then the Jacobian $\displaystyle{J(x,y)=\left[\frac{\partial (u_1,u_2)}{\partial (x,y)}\right]}$ does not vanish at the origin.

\end{cor}
\begin{proof}
An immediate corollary to previous result. Since every pluriharmonic map is harmonic.
\end{proof}

\section{Higher dimensional case}

In previous section, we studied Lewy's theorem for harmonic functions. In chapter [5], we will see that Lewy's theorem does not hold in $\mathbb{C}^n\geq 2$. So question comes whether Lewy's theorem holds for    pluriharmonic functions in $\mathbb{C}^n$. We will now study generalization of Lewy's theorem for pluriharmonic functions on $\mathbb{C}^2$. For $n>2$ this problem is open for pluriharmonic functions.\\
Let n be a fixed natural number.
Fix $1\leq j\leq n$.
For $z=(z_1,z_2\ldots, z_n) \in{\mathbb{C}^n}$,
we denote 
$$D_j=\displaystyle{\frac{\partial}{\partial z_j}}\ \mbox{and }\  \displaystyle{{\overline{D_j}}}=\frac{\partial}{\partial  \overline{z_j}},$$
where $\displaystyle{\frac{\partial}{\partial z_j}=\frac{1}{2}\left(\frac{\partial}{\partial x_j}-i\frac{\partial}{\partial y_j}\right)}$ and $\displaystyle{\frac{\partial}{\partial \bar{z_j}}=\frac{1}{2}\left(\frac{\partial}{\partial x_j}+i\frac{\partial}{\partial y_j}\right)}.$\\
With these notations, note that
a mapping $f=u+iv:G\subseteq \mathbb{C}^n\rightarrow \mathbb{C}$ is harmonic if $D_j \circ \overline{D_j}=0 \ \mbox{for all}\  1\leq j \leq n$.
\begin{defi}\textbf{(Pluriharmonic function)}
Let $\Omega \subset \mathbb{C}^n$ be open. And let $f$ be a real valued function defined on $\Omega$ such that $f\in{C^2(\Omega)}$. We say $f$ is pluriharomnic if $$D_j \circ \overline{D_k}f=0 \ \mbox{for all} \  j,k = 1,2,\ldots,n.$$
\end{defi}
\textbf{Remarks}\\
1. A complex valued function $f$ defined on a domain of $\mathbb{C}^n$ is said to be pluriharmonic if real and imaginary parts of $f$ are pluriharmonic. So, if we denote $f=u+iv$, then $f$ is pluriharmonic if $u$ and $v$ are pluriharmonic. That is $$D_j \circ \overline{D_k}u=0\ \mbox{ and} \  D_j \circ \overline{D_k}v=0\ \mbox{ for all} \ 1\leq j,k \leq n,$$ thus $D_j \circ \overline{D_k}f=D_j \circ \overline{D_k}u+iD_j \circ \overline{D_k}v=0$ for all $1\leq i,k \leq n$.\\
2. Let $f$ be an analytic function defined on $\Omega \subseteq \mathbb{C}^n$. Thus
$f$ is analytic in each variable, so we get that $\frac{\partial f}{\partial \bar{z_j}}=0$  $\mbox{for all} \  1 \leq j \leq n$. We have
$\overline{D_j}f=0$ $\ \mbox{for all} \  1 \leq j \leq n$ and therefore
$D_i \circ\overline{D_j}f=0$ $\ \mbox{for all} \  1 \leq i, j \leq n$.
Thus $f$ is complex pluriharmonic.
Therefore every analytic function is complex pluriharmonic.\\
3. We recall that a functin $f$ on a domain in $\mathbb{C}$ is analytic if and only if it is independent of $\bar{z}$ , that is  $$\displaystyle{\frac{\partial}{\partial \bar{z}} f=0}.$$
4. \textbf{Example of a pluriharmonic function which is not analytic}\\
Consider map $f:G\subseteq \mathbb{C} \rightarrow \mathbb{C}$ given by 
$$f(z)=\bar{z},$$
then clearly $f$ is not analytic.
But $\displaystyle{D_1\circ \bar{D_1}f=\frac{\partial}{\partial z} \frac{\partial}{\partial \bar{z}} f=0}$. Therefore $f$ is pluriharmonic but not analytic.\\
5. Let $f$ be a pluriharmonic function defined on $\Omega \subseteq \mathbb{C}^n$.
$\mbox{That is}, \ D_i\circ\bar{D_j}=0$ $\ \mbox{for all} \  1\leq i,j \leq n$.
In particular $D_i\circ \overline{D_i}f=0 \ \mbox{for all} \  1\leq i \leq n$.
And notice that
 \begin{align*}
 D_i\circ \overline{D_i}&=\frac{\partial}{\partial z_i}\frac{\partial}{\partial \bar{z_i}}\\
&=\frac{1}{2}\left(\frac{\partial}{\partial x_i}-i\frac{\partial}{\partial y_i}\right)\frac{1}{2}\left(\frac{\partial}{\partial x_i}+i\frac{\partial}{\partial y_i}\right)\\
&=\frac{1}{4}\left(\frac{\partial ^2}{\partial x_i ^2}+\frac{\partial ^2}{\partial y_i ^2}\right)\ \ \mbox{(since} \ f\in{C^2(\Omega)}) .
\end{align*}
Therefore  $$D_i\circ \overline{D_i}f=0$$
$$\left(\frac{\partial ^2}{\partial x_i ^2}+\frac{\partial ^2}{\partial y_i ^2}\right)f=0 \ \mbox{for all} \  i$$
$$ \left(\frac{\partial ^2}{\partial x_i ^2}+\frac{\partial ^2}{\partial y_i ^2}\right)u + i \left(\frac{\partial ^2}{\partial x_i ^2}+\frac{\partial ^2}{\partial y_i ^2}\right)v =0$$
 $$\left(\frac{\partial ^2}{\partial x_i ^2}+\frac{\partial ^2}{\partial y_i ^2}\right)u=0\ \mbox{
and} \left(\frac{\partial ^2}{\partial x_i ^2}+\frac{\partial ^2}{\partial y_i ^2} \right)v=0 \ \mbox{for all} \  i.$$
We see that $f$ is complex harmonic.
Therefore the class of pluriharmonic functions is contained in the class of complex harmonic functions.\\
6. \textbf{Example of a real valued harmonic map which is not pluriharmonic.}\\
Consider a map $u:G\subseteq \mathbb{C}^2 \rightarrow \mathbb{R}$ defined by
$$\displaystyle{u(z)=u(z_1,z_2)=x_1y_1x_2y_2},$$
then $u$ is harmonic but not pluriharmonic.
\begin{lem}
Let $\phi$ be a real valued function of class $\mathbb{C}^2$ in a domain $\Omega \subseteq \mathbb{C}^n$. Then $\phi$ is pluriharmonic if and only if for any complex line $L$, the restriction of $\phi$ to $L \cap\Omega$ is a harmonic function as a function of one variable on each component of $L\cap \Omega$.
\begin{proof}
Consider any complex line $L=\{ct+b:t\in{\mathbb{R}},\ c,b\in{\mathbb{C}^n}\}$, where $c\neq 0$ is direction of line $L$ passing through the point $b$. For $ct+b\in{\Omega}$, consider function $$\psi(t)=\phi(c_1t+b_1,\ldots, c_nt+b_n).$$
Then $$\frac{\partial ^2\psi}{\partial t\partial \bar{t}}(t)=\sum_{i,j=1}^{n}\frac{\partial ^2\phi}{\partial z_i\partial \bar{z_j}}(ct+b)c_i\bar{c_j},$$
which gives the desired result.
\end{proof}
\end{lem}
\begin{rem}
Let $h$ be a multivalued function defined on an open subset $U$ of $\mathbb{C}^n$. Then for a curve $\gamma$ in $U$, by $\triangle_{\gamma} \ h$, we denote the change in value of $h$ from starting point of $\gamma$ to end point of $\gamma.$ For example let $h$ be the argument function. Then for any circle $\gamma$ in $U$, value of $h$ at the starting point is $2\pi n$, for some $n\in{\mathbb{N}}$, and value of $h$ at the end point is $2\pi n+2\pi$. So the difference is given by $\triangle_{\gamma} h(z)=2\pi$, no matter what branch we  pick. For other multivalued functions, this may not be true.
\end{rem}
\begin{thm}
Let $h$ be a continuous one-one function of the neighborhood $U$ of the point 0 $\in{\mathbb{C}^2}$ into the space $\mathbb{R}^4$, $h=(h_1,h_2,h_3,h_4)$. If $h_1=\Re P+r$, where $P$ is a homogeneous holomorphic polynomial of degree $k\geq 1$ in $\mathbb{C}^2$ and $r=o(\left|z\right|^k)$, then $\mathsf{grad}\ h_1(0)\neq 0$.
\begin{proof}
Let $f=P+r$, without loss of generality we may assume that $P(z_1,0)\neq 0$, since this can be achieved by a non singular linear transformation. Let us denote by $K$ the cone $$\{z\in{U}:\left|P(z)\right|\leq c\left|z\right|^k\},$$
 where constant $c$ is chosen in such a way that the section of $K$ by an arbitrary plane $\{z_2=\mbox{constant}\}$ is connected. Let $\delta $ be so small that the section of $K$ by plane $\{z_2=\delta\}$ is contained in the section of $U$ by the same plane, that is $K\cap \{z_2=\delta\}\subseteq U\cap \{z_2=\delta\}$. Then for any closed curve $\gamma$ in this section, using argument principle,
\begin{equation}
\triangle_{\gamma}Arg\ f=2k\pi l,
\end{equation}
 where $l$ is an integer. Note that $K$ contains every zero of $P$ and therefore $\triangle_{\gamma} Arg \ P=2k\pi l=\triangle_{\gamma}Arg \ f$. Note that $f(z_0,0)=z_1^k(1+o(1)),$ and if $h_1(z_1,0)=0$ then $\Re \ f(z_1,0)=0$. 
In the section $\{z_2=0\}\cap V=\{z:|z_1|<\rho\}$, choose a point $(a_1,0)$ such that $\displaystyle{|a_1|<\frac{\rho}{2}}$ at which $f= h_1 >0$. (for this consider $\gamma$ as a circle of radius $\displaystyle{\frac{\rho}{2}}$, from $(3.1)$ we can choose such a point)\\
In a similar way, we can  decrease $\delta$ such that $h(a_1,z_2)>0$, where $|z_2|<\delta$. Denote the circle $\{\displaystyle{|z_1|=\frac{\rho}{2}},z_2=c\}$ by $\displaystyle{\gamma_{c}}$. Define 
 $$\displaystyle{\alpha (z)= Arg\ f(z),}$$ where $z=(a_1,z_2),|z_2|<\delta$. Choose a branch of argument for which $\displaystyle{|\alpha(z)|<\frac{\pi}{2}}$. Now we will extend this function to the points on the circle, $\gamma_c$. Take a point $(z_1,c)$ on the circle $\displaystyle{\gamma_{c}}.$ Let $\gamma_{z'}$ and $\gamma_{z"}$ be arcs of $\gamma_{c}$ joining $(a_1,c)$ and $(z_1,c)$ in the positive and negative directions respectively. Move along $\gamma_{c}$ in the positive direction until $$\displaystyle{\triangle_{\gamma_{z'}}Arg\ f+\alpha(a_1,c)<k\pi +\frac{\pi}{2}}$$ and in the negative direction until $$\displaystyle{\triangle_{\gamma_{z"}}Arg\ f+\alpha(a_1,c)>k\pi +\frac{\pi}{2}}.$$ We define $$\displaystyle{\alpha(z_1,c)=\triangle_{\gamma_z}Arg \ f(z)+\alpha\left(a_1,c\right),}$$ where $\gamma_z$ is either $\displaystyle{\gamma_{z'}}$ or $\displaystyle{\gamma_{z"}}.$ Now we will extend this function to the points of the set $\big[\left(\{z_2=c\}\cap\{|z_1|<\rho,|z_2|<\delta\}\right)\setminus \{h_1=0\}\big]$, which can be joined, without leaving the set, to the points of the set $\gamma_c \setminus \{h_1=0\}$. \\
 For this, let $\displaystyle{z\in\big[{\left(\{z_2=c\}\cap\{|z_1|<\rho,|z_2|<\delta\}\right)\setminus \{h_1=0\}}\big]}$ and $\lambda$ be the curve joining $z$ to the point $z_0\in{\gamma_c \setminus \{h_1=0\}}$, we define $$\alpha(z)=\alpha(z_0)+\triangle_{\gamma}Arg \ f.$$ First of all notice that if we fix point $z_0$, and if $\lambda_1$ and $\lambda_2$ are two curves joining $z$ and $z_0$, then since starting and end points of both $\lambda_1$ and $\lambda_2$ are same and also $\left|\triangle_{\lambda}Arg \ f\right|<\pi,$ we get that definition of $\alpha$ does not depend on choice of $\lambda$, if $z_0$ is fixed.  Now let us assume that $z$ can be joined to two points $z_0$ and $z'\in{\gamma_c}$ by curves $\lambda_0$ and $\lambda_1$ respectively. Take $\lambda_2=\lambda_0\cup \bar{\lambda_1}$, then $\lambda$ does not intersect $\{h_1=0\}$ and therefore $|\triangle_{\lambda_2}Arg \ f|<\pi$. Let $\lambda_3$ be a curve joining $z_0$ and $z'$ such that $\left|\triangle_{ \lambda_3}\right|\leq k\pi$. (we can choose such a curve since $\triangle_{\gamma_c}Arg \ f=2k\pi)$\\
Take $\lambda=\lambda_2\cup \lambda_3$, then $$\left|\triangle_{\gamma}Arg \ f\right|<k\pi+\pi.$$ And since $k\pi+\pi\leq 2k\pi,$ we ge that $\triangle_{\gamma}Arg \ f=0$ and thus we get that definition of $\alpha$ does not depend on the choice of the point $z_0.$\\
Take $U_1$ as an arbitrary neighborhood of zero contained in $V$. Then $U_1$ intersects every component of the set $\{(z_1,0):|z_1|<\rho,z_1\notin{L}\}$. Since $L$ divides the disk $|z_1|<\rho$ into $2k$ non intersecting open angles with vertices at 0, therefore we get that $\displaystyle{\{\left(z_1,0\right):|z_1|<\rho,z_1\notin{L}\}}$ consists of 2k componets. And $U_1$ intersects each one of them, therefore $U_1\setminus \{h_1=0\}$ consists atleast $2k$ connected components.\\
Notice that here $U_1$ was an arbitrary neighborhood of zero. In particular, take $U_1$ as the inverse image of sufficiently small sphere. Since the plane $h_1=0$ in $\mathbb{R}^4$ divides each sphere into exactly two components, and since $h$ is a homeomorphism, therefore we must have $2k=2$, that is, k=1. Hence we obtain that $P$ is a linear polynomial. Therefore $\mathsf{grad}\ P\neq 0.$ 
\end{proof}
\end{thm}
\begin{cor}
Let $h_i,\ \left(i=1,2,3,4\right)$ be pluriharmonic functions in the neighborhood $U$ of the origin in $\mathbb{C}^2$, such that the mapping $h=\left(h_1,h_2,h_3,h_4\right):U\to \mathbb{R}^4$ is one-one. Then the real Jacobian $J_{\mathbb{R}}(h)$ of $h$ does not vanish at the origin.
\begin{proof}
Since $h_1$ is pluriharmonic, therefore locally $h_1$ can be written as real part of an analytic function $f$. Considering taylor series expansion for $f$, and taking P to be homogeneous polynomial corresponding to the lowest degree terms in the expansion and $r=\Re \left(g-P\right)$. We see that $h_1=\Re P+r$. We see that $h_1$ is in the required form as in previous lemma. Now the result follows from the procedure as in Lemma $3.1.2.$
\end{proof}
\end{cor}
\begin{thm}
\textbf{(Inverse function theorem for domains in $\mathbb{C}^n$)}\\
Let $\Omega \subseteq \mathbb{C}^n$ be open and $f:\Omega \rightarrow \mathbb{C}^n$ be holomorphic. Suppose that $f'(z)$ is invertible at some point $z\in \Omega$ (\mbox{that is} \  $J_f(z) \neq 0$).
Then there are neighborhood U of z and V of $f(z)$ such that $f:U \rightarrow V$ is one-one, onto and $f^{-1} :V \rightarrow U$ is holomorphic.
\begin{proof}
We have $\Omega \subseteq \mathbb{C}^n$ is open,
$f:\Omega \rightarrow \mathbb{C}^n$ is holomorphic and $f'(z)$ is invertible at some point $z\in \Omega$.
Since $f'(z)$ is invertible as a linear operator on $\mathbb{C}^n$, it is also invertible as a linear operator on $\mathbb{R}^{2n}$.
Therefore by real version of inverse function theorem, we get that $f$ is one-one map from some neighborhood U of z onto a neighborhood V of $f(z)$, whose inverse is of class $C^1$.
Let us denote this inverse of $f$ by $g$. Then $g$ is of class $C^1$.
Now, the only point to be shown is that $g=(g_1,g_2,\ldots,g_n)$ is a holomorphic map.\\
Since $g$ is inverse of $f$, we get that
$g(f(z))=z\quad \ \mbox{for all} \  z\in U$.
Which implies that $$\quad g_i(f(z))=z_i \quad \ \mbox{for all} \  i; \quad 1\leq i \leq n.$$
Applying $ \overline{D_k}$ and using chain rule, we get 
$$\sum _{j=1}^{n}( \overline{D_k} \overline{f_j}(z))( \overline{D_j}g_i(f(z)))=0.$$
If we denote $c_{ik}=\displaystyle{\sum _{j=1}^{n}\left( \overline{D_k} \overline{f_j}(z)\right)\left( \overline{D_j}g_i(f(z))\right)}$, then matrix $C=AB$, where $A=\left(a_{kj}\right)=\left(\overline{D_k} \overline{f_j}(z)\right)$ and $B=\left(b_{ji}\right)=\left(\overline{D_j}g_i(f(z))\right),$
and since $J_{f}(z)\neq 0$, we get that $A$ is invertible, which gives that $B=0$.
That is, $\overline{D_j}g_i(f(z))=0$ for all $j=1,\ldots, n$, which gives that $g_i$ is analytic. And since $g_i$ was arbitrary, we get that each component of $g$ is analytic. Hence $g$ is analytic.                     
\end{proof}
\end{thm}
Converse of the above result is also true, we will need some results to prove it. So we will first prove these results and then we will prove converse of the above theorem.
\begin{thm}
\textbf{(Rad\'{o}s theorem)} Let $\Omega \subseteq \mathbb{C}^n$ be a region. Assume that $f:\Omega \to \mathbb{C}$ is continuous and $f$ is analytic in an open subset of $\Omega$ where $f(z)\neq 0$. Then $f$ is analytic on $\Omega$.
\begin{proof}
Without loss of generality we may assume that $f$ is continuous on $\overline{D}$, and analytic on $D$, where $D$ is open unit ball. Let $E$ be the set given by $$E=\{z:f(z)=0\},$$ 
then $f$ is analytic on $D\setminus E$ and $|f|<1.$ Note that restriction of $f$ on boundary $\mathbb{T}=\partial D$ of $D$ is continuous. So we can extend this to a complex harmonic function on $D$ by Poisson integral formula. Let us denote this extension by $g$. Let $\alpha$ be any positive constant. Define $$\phi=\Re(f-g)+\alpha \log |f|.$$ Since $f$ is analytic, we get that $\log|f|$
is harmonic in $D\setminus E$. Also $\Re f$ is harmonic and $\Re g$ is harmonic, since $g$ is complex harmonic. Thus  $\phi$ is harmonic in $D\setminus E$.\\
As $z\to z_0\in{E}$; $z\in{D\setminus E}$ $$\phi(z)\to \infty.$$
When $z\to e^{i\theta}\in{\mathbb{T}}$ $$\phi(z)\to \alpha \log \left|f\left(e^{i\theta}\right)\right|<0,$$
(here we are using that $\log|f|<0$ for $|f|<1$ )\\
therefore we get by maximum priciple for harmonic functions that $$\phi(z)\leq 0.$$ 
Letting $\alpha \to 0$, we get that $$\Re (f-g)\leq 0\ \mbox{in} \ D\setminus E.$$
Again by taking $\alpha$ to be negative constant and repeating the same argument, we will get that $$\Re(f-g)\geq 0\ \mbox{in} \ D\setminus E.$$ Thus $\Re f=\Re g$. Now considering imaginary part of $f$, $g$ and repeating the same argument, we get $\Im f=\Im g$. Thus $$f(z)=g(z)\ \mbox{on} \ D\setminus E.$$ 
By continuity we obtain 
$$f(z)=g(z)\ \mbox{on} \ \overline{D\setminus E}.$$ 
Also, since $g$ is harmonic extension of $f$. So $f=g$ on $\mathbb{T}$. Thus We have $$f(z)=g(z)\ \mbox{on} \ \mathbb{T} \cup \overline{D\setminus E}.$$
Also $f=0$ on $E$ and therefore $f=0$ on $\partial E$, $g=f=0$ on $\partial E$. So by maximum principle for harmonic functions we get that $g=0$ on $E$. Hence $$f\equiv g \ \mbox{on} \ \overline{D}.$$
\end{proof}
\end{thm}
For a region $\Omega$ in $\mathbb{C}^n$, we use the notation $H(\Omega)$ for the collection of functions analytic on $\Omega$.
\begin{defi}$(\textbf{\mbox{$H^{\infty}$-Removable}})$
Let $\Omega$ be a region in $\mathbb{C}^n$. A relatively closed subset $E$ of $\Omega$ is said to be $H^{\infty}$- removable in $\Omega$ if every bounded funtion $f\in{H(\Omega \setminus E)}$ has an extension $F\in{H(\Omega)}$. 
\end{defi}
\begin{lem}~\cite{rudin3}
If $\Omega$ is a region in $\mathbb{C}^n$, $g\in{H(\Omega)}$, if $f$ is not identically zero and $$E=\{z\in{\Omega}:g(z)=0\}.$$ Then $E$ is $H^{\infty}-$removable in $\Omega$.
\end{lem}
\begin{thm}
If $\Omega \subseteq \mathbb{C}^n$ and $f:\Omega \rightarrow \mathbb{C}^n$ is holomorphic and one-one, then $f'(z)$ is invertible $\ \mbox{for all} \  z\in{\Omega}$.
\begin{proof}
Assume that $f$ is analytic and is one-one. We need to show that Jacobian $J_{f}$ of $f$ has no zeroes in $\Omega$. Since $f$ is one-one, $$f^{-1}(w)=\phi \ \mbox{or} \ f^{-1}(w)={z_w}.$$ 
That is $f^{-1}(w)$ is compact for all $w\in{\mathbb{C}^n}$, therefore we get that $f$ is an open map.\\
Let $\Omega '=f(\Omega)$, then $$f:\Omega \subseteq{\mathbb{C}^n}\to f(\Omega)$$ is one-one, onto and continuous. Also, since $f$ is an open map, $f^{-1}$ is continuous. Thus $$f:\Omega \subseteq{\mathbb{C}^n}\to f(\Omega)$$ is a homeomorphism.\\
We define a map $g$ on $\Omega '=f(\Omega)$ as $$g(w)=g(f(z))=J(f^{-1}(w))=J(f^{-1}(f(z))).$$
Then $g$ is continuous, since if we take a sequence $\{w_n\}$ in $\Omega'$ such that 
$w_n\to w$, then
 $$f^{-1}(w_n)\to f^{-1}(w)$$
 $$A_n(f^{-1}(w_n))\to A(f^{-1}(w))$$
 $$detA_n(f^{-1}(w_n))\to detA(f^{-1}(w)),$$ (where $A_n$ and $A$ are Jacobian matrices
  corresponding to $f^{-1}(w_n)$ and $f^{-1}(w)$ respectively)\\
 $\mbox{That is} \ g(w_n)\to g(w)$.
 Thus $g$ is continuous. Let M be the set given by $$M=\{z:J(z)=0\}.$$
Then since determinant function is continuous and $M$ is inverse image of the set $\{0\}$, $M$ is closed in $\Omega '$. Also, since $f^{-1}$ is continuous, we get that $f$ is closed. Therefore $f(M)$ is closed and $\Omega \setminus f(M)$ is open. Thus $g$ is continuous on an open subset $\Omega ' \setminus f(M)$ of $\Omega'$. So we get that $g$ is analytic on $\Omega ' \setminus f(M)$.\\ 
Now, notice that when $w\in{f(M)=\{f(z):z\in{M}\}}=\{f(z):J(z)=0\}$,
then $$g(w)=J(f^{-1}(w)); \ J(z)=0.$$
Also when $w\notin{f(M)}, \ \mbox{we have} \  w\neq f(m)\ \mbox{for any }\ m\in{M}$.\\
That is $w\neq f(z)\ \mbox{such that } \ J(z)=0$.
That is $ f^{-1}(w)\neq z\ \mbox{such that } \ J(z)=0$,
and thus we see that $J(f^{-1}(w))=g(w)\neq 0$.
Therefore $g=0$ exactly on $f(M)$. Thus by Rad\'{o}'s theorem, we get that $g$ is analytic on $\Omega '$.
Therefore using lemma $ 3.2.2$ we get that $f(M)$ is $H^{\infty}$ removable. That is every bounded function $h$ analytic on $\Omega '\setminus f(M)$ extends to an analytic function $h'$ on $\Omega '$.
Observe that whenever $w\in{\Omega '\setminus f(M)}$, $J(z)\neq 0$,
$f\ \mbox{is} \ \mbox{one to one in a neighborhood of z onto a neighborhood of } \ w. \ \mbox{So, we get that}\ 
 f^{-1} $ exists locally at $w$. Therefore $f^{-1}$ exists whenever $w\in{\Omega '\setminus f(M)}$ and $f(M)$ is $H^{\infty}$-removable. Thus $f^{-1}$ is analytic on $\Omega'$,
and $f^{-1}(f(z))=z$, which gives that $J_{f^{-1}}J_f=I,$ where $I$ is an $n\times n $ matrix. Therefore we see that the Jacobian matrix is invertible and hence we get that $J_f(z)$ has no zeroes in $\Omega$.
 
\end{proof}
\end{thm}

\chapter{Injectivity of harmonic extensions}

Poisson inegral formula can be considered as a way to find a real valued harmonic function with certain boundary conditions. It has a lot of applications. For example, if the boundary conditions give the temperature of a perfectly insulated plate then the harmonic extension gives the temperature of interior of the plate. We will begin the chapter by introducing Poisson integral formula and discussing some of its properties. This chapter is based on the work done by Choquet ~\cite{choquet4}, Kneser ~\cite{kneser4} and Rad\'{o} ~\cite{rado4}. 
\begin{defi}
A function $f:\Omega \subseteq \mathbb{C} \to \mathbb{C} $ is said to homeomorphism if it is bijective and $f$ and $f^{-1}$ are both continuous.
\end{defi}
\begin{defi}
The function $P_r(\theta )=\sum _{n=- \infty}^{\infty} r^{|n|}e^{ i n \theta}\ \ \  0\leq r <1, -\infty <\theta < \infty$ is called Poisson kernel.
\end{defi}
\begin{rem}
Let $z=re^{ i \theta }\ \ \ 0\leq r <1$,
consider \begin{align*}
\frac{1+r e^{ i \theta }}{1-re^{ i \theta}}&=\frac{1+z}{1-z}\\
&=(1+z)(1+z+z^2+\ldots)\\
&=1+2\sum_1^{\infty}z^n\\
&=1+2\sum_1^{\infty}r^ne^{ i n\theta }.
\end{align*}
And therefore $\displaystyle{P_r(\theta)=\Re \left(\frac{1+r e^{ i \theta }}{1-re^{ i \theta}}\right).}$\\
\end{rem}
\begin{rem}
We can write 
\begin{align*}
\frac{1+r e^{ i \theta }}{1-re^{ i \theta}}&=\frac{1+r e^{ i \theta }}{1-re^{ i \theta}} .\frac{1-r e^{- i \theta }}{1-re^{- i \theta}}\\
&=\frac{1+re^{ i \theta}-re^{- i \theta}-r^2}{\left|1-re^{ i \theta}\right|^2}\\
&=\frac{1+2 i r \sin \theta -r^2}{1+r^2-2r\cos \theta}.
\end{align*}
Therefore, \begin{align*}
P_r(\theta)&=\Re \left(\frac{1+r e^{ i \theta }}{1-re^{ i \theta}}\right)\\
&=\frac{1 -r^2}{1+r^2-2r\cos \theta},
\end{align*}
and hence $P_r$ is a periodic function of period $2\pi$.
\end{rem}
\begin{rem}
For every $\delta >0$,
$P_r(\theta) \to 0$ uniformly on $\delta \leq |\theta |\leq \pi$ as $r\to 1^-$.
\end{rem}
\begin{thm}
Let $D=\{z:|z|<1\}$ be unit ball in $\mathbb{C}$ and suppose that $f:\partial D\to \mathbb{R}$ is continuous. Then there is a unique function $F: \overline{D}\to \mathbb{R}$ which equals $f$ on $\partial D$ and is harmonic in $D$.
\begin{proof}
Define $F$ on $ \overline{D}$ as \\
\[ F(re^{ i \theta })=
\begin{cases}
\frac{1}{2\pi} \int_{-\pi}^{\pi} P_r(\theta -t)f(e^{ i t})dt\ \ \  0\leq r <1,\\
f(e^{ i \theta })\ \ \ \ \ \ \ \ \ \  \ \ \ \ \ \ \ \  \ \ \ \ \ \ \ r=1,
\end{cases}
\]
then clearly $F$ equals $f$ on $\partial D$. We claim that
$F$ is harmonic on $D$.
For this it suffices to show that $F$ is real part of an analytic function.\\
For $0\leq r <1$
\begin{align*}
F(re^{ i \theta })&=\frac{1}{2\pi}\int_{-\pi}^{\pi} P_r(\theta -t)f(e^{ i t})dt\\
&=\frac{1}{2\pi}\int_{-\pi}^{\pi} \Re \ \left (\frac{1+r e^{ i (\theta -t) }}{1-re^{ i (\theta -t)}}\right)f(e^{ i t})dt\\
&=\frac{1}{2\pi}\Re \left[\int_{-\pi}^{\pi}  \frac{1+r e^{ i (\theta -t) }}{1-re^{ i (\theta -t)}}f(e^{ i t})\right ]dt. 
\end{align*}
Thus $F$ is harmonic on $D$ and therefore continuous on $D$.
Now, we will show that $F$ is continuous on $\partial D$.
Firstly consider point $z=1$ on $\partial D$.\\
Let $\epsilon >0$ be given,
we need to show that $F$ is continuous at point $1$.
For that we will show that, \\
for given $\epsilon >0$, there is $\rho \in{[0,1]}$ and $\delta >0$ such that 
$$\ \mbox{for all} \  \rho <r \leq 1\ \mbox{ and } \ |\theta |<\delta\ \ 
|F\left(re ^{ i \theta}\right)-F(1)|< \epsilon.$$
Consider, 
\begin{align*}
F(re ^{ i \theta})-F(1)&=F(re ^{ i \theta})-f(1) \ \ \ \ \ \ \ \ \mbox{(since} \  F=f \ \mbox{on} \  \partial D)\\
&=\frac{1}{2\pi}\int_{-\pi}^{\pi}P_r(\theta -t)f(e ^{ i t}) dt-\frac{1}{2\pi}\int_{-\pi}^{\pi}P_r(\theta -t)f(1) dt\\
&=\frac{1}{2\pi}\int_{-\pi}^{\pi}P_r(\theta -t)\left(f(e ^{ i t})-f(1)\right) dt.
\end{align*}
Now since $f$ is continuous at $z=1$, so corresponding to $\epsilon >0$, there exists $ \delta >0$ such that
$$\left|f(e^{ i t})-f(1)\right|<\frac{\epsilon}{3}, \ \mbox{whenever $|t|<\delta$}.$$
Therefore,
\begin{align*}
|F(re^{ i \theta})-F(1)|&\leq \frac{1}{2\pi}\left|\int_{|t|<\delta}P_r(\theta -t)\left(f(e^{ i t})-f(1)\right)\right|+I_1\\
&\leq \frac{\epsilon}{3}+I_1,\ \mbox{ where} \ I_1=\frac{1}{2\pi}\int_{\delta \leq |t|\leq \pi}P_r(\theta -t)\left(f(e^{ i t})-f(1)\right)dt.
\end{align*}
Thus we have $$|I_1|\leq M \frac{1}{\pi}\int_{\delta \leq |t|\leq \pi} P_r(\theta -t)dt,$$ where $M=\sup \{\left|f(e^{ i t})\right|: -\pi\leq t\leq \pi \}$. (supremum exists, since $f$ is continuous on $\partial D $, which is compact)\\
Also, for $|\theta |<\frac{\delta }{2}$ and $|t|<\delta$, 
$|\theta -t| \geq \frac{\delta}{2}$ and since for $|\alpha| >\frac{\delta}{2}$, $P_r(\alpha )\to o$ uniformly on $\delta \leq |t|\leq \pi$ as $r\to 1^-$, therefore
there exists $ \rho \in{[0,1]}$ such that $\ \mbox{for all} \  \rho <r<1$
$$\left|P_r(\alpha )\right|< \frac{\epsilon}{6M}.$$
Therefore we get $|I_1| \leq \frac{\epsilon}{3},$
and hence for $\rho <r<1$ and $\displaystyle{\left|\theta\right|<\frac{\delta}{2}}$,
$$\left|F(re^{ i \theta})-F(1)\right|<\epsilon,$$ which gives that $u$ is continuous at $z=1$.
Now take any point  $e^{ i \beta},\ \  \beta \in{(-\pi,\pi)}$ on $\partial D$ other than the point $z=1$.\\
Define a map $G$ on $ \overline{D}$ as
$$G(z)=F(e^{ i \beta }z).$$
Then for $z\in{\partial D},$
\begin{align*}
G(z)&=f(e^{ i \beta }z)\\
&=g(z), \ \mbox{(say)},
\end{align*}
where $g$ is continuous on $\partial D$.
Also for $z\in{D}$,
\begin{align*}
G(z)&=F(e^{ i \beta }z)\\
&=F\left(e^{ i \beta }re^{ i \theta}\right)\\
&=F\left(r e^{ i (\beta +\theta)}\right)\\
&=\frac{1}{2\pi}\int_{-\pi}^{\pi}P_r(\theta +\beta -t)f(e^{ i t})dt\\
&=\frac{1}{2\pi}\int_{-\pi-\beta}^{\pi -\beta}P_r(\theta -s)f\left(e^{ i (s+\beta)}\right)ds \ \mbox{(putting} \  s=t-\beta)\\
&=\frac{1}{2\pi}\int_{-\pi-\beta}^{\pi -\beta}P_r(\theta -s)g(e^{ i s})ds\\
&=\frac{1}{2\pi}\int_{-\pi}^{\pi }P_r(\theta -s)g(e^{ i s})ds \ \ \ \ \mbox{(since} \  P_r \ \mbox{is periodic with period} \  2\pi).
\end{align*}
So, $G$ has same structure as that of $F$, 
therefore proceeding as in case of $F$, we get that $G$ is continuous at the point $z=1$.
And since $G(1)=F(e^{ i \beta})$,
we get that $F$ is continuous at the point $w=e^{i\beta}$. And since $w=e^{i\beta}$ was an arbitrary point, we get that $f$ is continuous on $\partial D$.
\end{proof}
\end{thm}
\begin{rem}
Note that since
\begin{align*}
\frac{1+re^{i(\theta-t)}}{1-re^{i(\theta-t)}}&=\frac{1-r^2+2ir\sin \theta}{|1-re^{i(\theta -t)}|^2},\ \mbox{(after rationalising)}\\
&=\frac{1-r^2+2ir\sin \theta}{|e^{it}-re^{i \theta}|^2},
\end{align*}
therefore Poisson integral formula for $f$ takes the form
\begin{align*}
F(z)&=F(re^{i\theta})\\
&=\frac{1}{2\pi}\Re \left[\int_{-\pi}^{\pi}  \frac{1+r e^{ i (\theta -t) }}{1-re^{ i (\theta -t)}}f(e^{ i t})\right ]dt\\
&=\frac{1}{2\pi}\Re \int_{0}^{2\pi}\frac{1-r^2}{|e^{it}-z|^2}f(e^{it})dt.
\end{align*}
\end{rem}
For a continuous complex valued function $f$ on unit circle in $\mathbb{C}$, consider real and imaginary parts of $f$, say $f_1$ and $f_2$ respectively. Then $f_1, f_2:D\subseteq \mathbb{C}\to \mathbb{R}$ are continuous. Therefore applying Poisson integral formula to each of them, we will get harmonic extensions $F_1,F_2$ of $f_1$ and $f_2$ respectively, which agrees with them on the circle, 
given by $$F_j(re^{ i \theta})=\frac{1}{2\pi}\int_{\pi}^{\pi} P_r(\theta -t)f_j(e^{ i t})dt \quad (j=1,2).$$
Take $F=F_1+ i F_2$,
\begin{align*}
F(re^{ i \theta})&=\frac{1}{2\pi}\int_{\pi}^{\pi} P_r(\theta -t)f_1(e^{ i t})dt+ i \frac{1}{2\pi}\int_{\pi}^{\pi} P_r(\theta -t)f_2(e^{ i t})dt\\
&=\frac{1}{2\pi}\int_{\pi}^{\pi} P_r(\theta -t)(f_1(e^{ i t})+f_2(e^{ i t}))dt\\
&=\frac{1}{2\pi}\int_{\pi}^{\pi} P_r(\theta -t)f(e^{ i t})dt. 
\end{align*}
Also $F$ agrees with $f$ on the circle, since $F_1,F_2$ agrees with $f_1$ and $f_2$ resp.

So we see that a complex valued continuous function can be extended to a complex harmonic function on $ \overline{D}$, which agrees with the given function on the circle.

Now comes the question of injectivity of this harmonic extension. This question was posed by Rad\'{o} and was proved by kneser ~\cite{kneser4} in 1926. In 1945, Choquet ~\cite{choquet4}, supplied a proof with different approach. On their names, this result has been given the name Rad\'{o}-Kneser-Choquet theorem.\\
\begin{lem}
Let $g$ be a real valued function, harmonic on unit disk $D$ in $\mathbb{C}$ and continuous on $\partial D$. Assume that $g$ has the property that after a rotation of co-ordinates, $g\left(e^{it}-e^{-it}\right)\geq 0$ on $\left[0,\pi \right]$ and $g\left(e^{it}-e^{-it}\right)>0$ on a subinterval $[a,b]$ with $0\leq a\leq b \leq \pi$. Then $\displaystyle{\frac{\partial g}{\partial z}}\neq 0$ on $D$.
\begin{proof}
For proving $\displaystyle{\frac{\partial g}{\partial z}}\neq 0$, it suffices to show that $g_z(0)\neq 0$. Since if $z_0\in{D}$ is any point, then consider the map $\phi(z)=\frac{z_0-z}{1-\bar{z_0} z}$, then $\phi$ is a self homeomorphism of $D$ with $\phi(0)=z_0$ and $\phi '(z)\neq 0$ on $D$. Let us consider map $h(z)=g\circ \phi(z)$. Then using remark $1.7$, we obtain that $h$ is harmonic. Also it can be easily checked that $h\left(e^{it}-e^{-it}\right)\geq 0$ on $\left[0,\pi\right]$ and $h\left(e^{it}-e^{-it}\right)>0$ on a subinterval $[a,b]$ with $0\leq a\leq b \leq \pi$. Also $$h_z(z)=g_z\left(\phi(z)\right)\phi '(z).$$ (here we are using that $\phi_{\bar{z}}\equiv 0$, since $\phi$ is analytic)\\
Therefore $h_z(0)=g_z(\phi(0))\phi '(0)$. Thus if we prove $g_z(0)\neq 0$, then using the same argument for $h$, we get $h_z(0)\neq 0$, which give that $g_z(z_0)\neq 0$. Thus it suffices to show that $g_z(0)\neq 0$.\\
We know
\begin{align}
g(z)&=\frac{1}{2\pi}\int_0^{2\pi}\frac{1-|z|^2}{|e^{it}-z|^2}g(e^{it})dt\\
&=\frac{1}{2\pi}\int_0^{2\pi}\frac{1-z\bar{z}}{(e^{it}-z)(e^{-it}-\bar{z})}g(e^{it})dt.
\end{align}
Now, since 
\begin{align}
\frac{\partial }{\partial z} \left( g(e^{it})\frac{1-z\bar{z}}{(e^{it}-z)(e^{-it}-\bar{z})}\right)&=\frac{ g(e^{it})}{e^{-it}-\bar{z}}\frac{\partial }{\partial z}\left(\frac{1-z \bar{z}}{e^{it}-z}\right)\\
&=\frac{ g(e^{it})}{e^{-it}-\bar{z}}\frac{e^{it\left(e^{-it}-\bar{z}\right)}}{\left(e^{it}-z\right)^2}\\
&=g(e^{it})\left(\frac{e^{it}}{\left(e^{it}-z\right)^2}\right).
\end{align}
Thus using $4.2$ and $4.5$, we obtain $$g_z(0)=\frac{1}{2\pi}\int_0^{2\pi}g(e^{it})e^{-it}dt,$$ which gives that
\begin{align*}
\Im g_z(0)&=\Im \left(\frac{1}{2\pi}\int_0^{2\pi}g(e^{it})e^{-it}dt\right)\\
&=-\frac{1}{2\pi}\int_0^{2\pi}g(e^{it}) \sin t \ dt\\
&=-\frac{1}{2\pi}\int_0^{\pi}g(e^{it}) \sin t \ dt+\frac{1}{2\pi}\int_{-\pi}^{0}g(e^{it}) \sin t \ dt\ \mbox{(since $g$ and $\sin$ are periodic with period $2\pi$)}\\
&=-\frac{1}{2\pi}\int_0^{\pi}g(e^{it}) \sin t \ dt-\frac{1}{2\pi}\int_{0}^{\pi}g(e^{-it}) \sin t \ dt\\
&=-\frac{1}{2\pi}\int_0^{\pi}\left(g(e^{it})-g(e^{-it})\right) \sin t \ dt,
\end{align*}
thus we get that $\Im g_z(0)\neq 0$. Since $g(e^{it})-g(e^{-it})\geq 0$ with strict inequality on a subinterval of $\left[0,\pi\right]$ and $\sin$ is non-negative on $\left[0,\pi\right]$. Hence $g_z(0)\neq 0$.
\end{proof}
\end{lem}
\begin{thm}
$\left(\mbox{\textbf{Rad\'{o}-Kneser-Choquet Theorem}}\right)$ Let $f$ be an orientation preserving and univalent on unit circle $\partial D=\{z:|z|=1\}$ in $\mathbb{C}$. Let $F$ be complex harmonic extension of $f$. Then $F$ defines sense preserving, univalent, onto function on $\{z:|z|<1\}$.
\begin{proof}
Without loss of generality we may assume that $f$ runs around $\partial D$ counterclockwise (otherwise we may consider conjugate of $f$). We first show that $F$ is locally univalent. Assume that $f$ is not locally univalent at $z_0\in D$. Thus by corollary $3.1.1$ the matrix 
$\begin{bmatrix}
u_x&v_x\\
u_y&v_y
\end{bmatrix}$ has determinant zero at $z=z_0$. Thus 
$$\begin{bmatrix}
u_x&v_x\\
u_y&v_y
\end{bmatrix}\begin{bmatrix}
a\\
b
\end{bmatrix}=0$$ for some $(a,b\neq 0)$ at $z=z_0$. Thus we get $$au_x+bv_x=0\ \mbox{and}$$
$$au_y+bv_y=0.$$ Therefore if we denote the function $au+bv$ by $g$, then $g_z(z_0)=0$. Note that $g$ satisfies hypothesis of lemma $ 4.0.2$. Therefore we get a contradiction. Thus we obtain that $f$ is locally univalent. Also, since $f$ is orientation preserving on $\partial D$, we get that $F$ is orientation preserving throughout $D$. Now we show $f$ is univalent. Assume that $F(z_1)=F(z_2)$ for $z_1,z_2\in D$. Then the map $F(z)-F(z_2)$ has two zeroes in $D$, which gives that winding number of $F(z)-F(z_2)$ about the origin is $2$. Which is a contradiction to the univalence of $f$ on $\partial D$ (here we are using that $f(z)-F(z_2)$ is univalent, since $f$ is univalent). Thus $F$ is univalent. Surjectivity of $F$ follows from lemma $5.2.1$. 
\end{proof}
\end{thm}

\chapter{Counterexamples in higher domains}
This chapter is based on the work done by R. S. Laugesen and J. C. Wood. It is shown that Rad\'{o}-Kneser-Choquet theorem fails to hold in $\mathbb{C}^n,n\geq 2$ by explaining construction of an example given by Laugesen. (See ~\cite{laug5} and ~\cite{wood5}).

\section{Jacobian problem}
In chapter $3$, we have seen that the Jacobian of a one-one harmonic map on a domain of $\mathbb{R}^2$ can not vanish. Here we will see that this no longer holds in $\mathbb{R}^n;n\geq 3$. This counterexample was given by J. C. Wood in his thesis ~\cite{wood15}.
\begin{thm} There exists a harmonic function on $\mathbb{R}^3$ such that $f$ is one-one but its Jacobian vanishes.
\begin{proof}
Consider map 
$f:\mathbb{R}^3\rightarrow \mathbb{R}^3$ given by
$$f(x,y,z)=(x^3-3xz^2+yz,y-3xz,z).$$
We will show that each component $f_i(x,y,z) (i=1,2,3)$ of $f$ is harmonic.
\begin{align*}
\triangledown^2 f_1(x,y,z)&= \frac{{\partial}^2 f_1}{\partial x^2}+\frac{\partial^2 f_1}{\partial y^2}+\frac{\partial^2 f_1}{\partial z^2}\\
&=\frac{\partial(3x^2-3z^2)}{\partial x}+\frac{\partial (z)}{\partial y}+\frac{\partial (y-6xz)}{\partial z}\\
&=6x+0-6x\\
&=0.
\end{align*}
Thus $f_1$ is harmonic.
Now, 
\begin{align*}
\triangledown ^2f_2(x,y,z)&=\frac{{\partial}^2 f_2}{\partial x^2}+\frac{\partial^2 f_2}{\partial y^2}+\frac{\partial^2 f_2}{\partial z^2}\\
&=\frac{\partial (-3z)}{\partial x}+\frac{\partial (1)}{\partial y}+\frac{\partial (-3x)}{\partial z}\\
&=0.
\end{align*}
Which implies that $f_2$ is harmonic.\\
Similarly, 
\begin{align*}
\triangledown ^2f_3(x,y,z)&=\frac{{\partial}^2 f_3}{\partial x^2}+\frac{\partial^2 f_3}{\partial y^2}+\frac{\partial^2 f_3}{\partial z^2}\\
&=0.
\end{align*}
Thus $f_3$ is harmonic.\\
Since each of the component is harmonic, we get that $f=(f_1,f_2,f_3)$ is harmonic.\\
Now we are going to show that $f$ is one-one.
Let us assume that
$f(x,y,z)=f(a,b,c)$ for $(x,y,z)$ and $(a,b,c)$ $\in{\mathbb{R}^n},$ which gives $z=c,\  y-3xc=b-3ac$  and $x^3-3xc^2+yc=a^3-3ac^2+bc$ which further implies that
$ x^3-a^3=3c^2(x-a)-(y-b)c=0$. 
Thus $ x=a$  (since $g(x)=x^3$ on $\mathbb{R}$ is one-one)
and $y-3xc = b-3ac$ implies $y=b$.
Thus $(x,y,z)=(a,b,c).$ Hence we get that $ f$ is one-one.
Finally we will show that Jacobian of $f$ vanishes.\\
Jacobian matrix of $f$ is given by
\begin{equation}\notag
\begin{bmatrix}
\frac{\partial f_1}{\partial x}&\frac{\partial f_1}{\partial y}&\frac{\partial f_1}{\partial z}\\
\frac{\partial f_2}{\partial x}&\frac{\partial f_2}{\partial y}&\frac{\partial f_2}{\partial z}\\
\frac{\partial f_3}{\partial x}&\frac{\partial f_3}{\partial y}&\frac{\partial f_3}{\partial z}
\end{bmatrix}
=\begin{bmatrix}
3x^2-3z^2&z&y-6xz\\
-3z&1&-3x\\
0&0&1
\end{bmatrix}.
\end{equation}
So Jacobian of $f$ is given by determinant of this matrix, which is 
$3x^2$, and it vanishes on the plane $\{x = 0\}$.
\end{proof}
\end{thm}
\begin{rem}This example can be trivially extended to form a counterexample in $\mathbb{R}^n,n\geq 3$ . We present this counterexample in the next theorem.
\end{rem}
\begin{thm}There exists a harmonic function on $\mathbb{R}^n,n\geq 3$ such that $f$ is one-one but its Jacobian vanishes.
\begin{proof}
Consider $f:\mathbb{R}^n\to \mathbb{R}^n$ defined as
$$f(x_1,x_2,\ldots,x_n)=(x_1,x_2,\ldots,x_{n-3},{x_{n-2}}^3-3x_{n-2}{x_n}^2+x_{n-1}x_n,x_{n-1}-3x_{n-2}x_n,x_n).$$
Then as in case of $\mathbb{R}^3$, each of the components is harmonic, giving that $f$ is harmonic. Also it can be easily verified that $f$ is one-one.
It's Jacobian is given by $$3{x_{n-2}}^2.$$
Which vanishes on the plane $$\{x_{n-2}=0\}.$$
\end{proof}
\end{thm}

\section{Poisson extension problem}
By Poisson integral formula each real valued continuous function on a circle in $\mathbb{C}$ can be extended to a harmonic function on the disc, which agrees with the given map on the circle. For a given homeomorphism of the unit circle onto itself, it can be extended to a complex valued harmonic function which maps unit circle onto itself, by considering real and imaginary parts of the function and applying Poisson integral formula to them. As seen in the previous chapter, this extension is injective due to Rad\'{o}-Kneser-Choquet theorem. In this chapter, we will see that this result fails to hold in $\mathbb{R}^n,n\geq 3$ and consequently in $\mathbb{C}^n,n\geq 2$.
\begin{lem}
If $f$ is a map which maps the sphere $S^{n-1}$ in $R^n$ homeomorphically onto itself, then the harmonic extension $F$ of $f$ maps the unit ball $B^n$ onto itself.
\begin{proof}
Let $n\geq 3$ and 
let $f$ be a self-homeomorphism of $S^{n-1}$.
We claim that $F$ maps $B^{n}$ onto $B^{n}.$
Let $x\in B^{n}$, by post rotating both $F$ and $f$, we may assume that
$$F_{1}(x)=|F(x)|,F_{2}(x)=0,\ldots,F_{n}(x)=0.$$
By strong maximum principle,
which states that if $h$ is a non constant harmonic map on a bounded domain $U$ which has an extension to $\overline{U} $, then $h$ satisfies $$\lvert h(x)\rvert < \max \left\{\lvert h(y) \rvert : y \in \partial U\right\},$$ 
we get that $$\left|F_1(x)\right|<\max \{F_1(x):x\in{S^{n-1}}\}.$$
That is $\left|F_1(x)\right|< \max f_1=1.$
Thus $\left|F(x)\right|<1$,
which implies that $F $ maps $B^{n}$ into $B^{n}$.
We will now show that $F$ maps $B^{n}$ onto itself. On contrary,
suppose there is 
\begin{equation}
y\in{B^n\setminus F(B^n)}.
\end{equation}
Define a map
$\Psi : \overline{B^n}\rightarrow S^{n-1}$ as
$$\Psi (x)=\frac{F(x)-y}{\left|F(x)-y\right|}.$$ Using $(5.1)$ we see that $\Psi$ is continuous.\\
Also note that $$\left|\Psi(x)\right|=1\quad \ \mbox{for all} \  \quad x\in{B^n\setminus F(B^n)}.$$
Let us denote $$ \psi=\Psi\mid_{S^{n-1}}:S^{n-1}\rightarrow S^{n-1},$$
$\mbox{that is} \ $ $$\psi(x)=\frac{f(x)-y}{|f(x)-y|}.$$  (since $F=f$ on $S^{n-1}$)\\
Then $\psi$ is continuous, since restriction of a continuous map is continuous. For injectivity of $\psi$ note that any ray from $y$ intersects $S^{n-1}$ in exactly one point. Therefore if $\psi(x)=\psi(x')$, then we will get that $f(x)$ and $f(x')$ lie on the same ray. And since $f(x)$ and $f(x')$ lie on the sphere, we get that $f(x)=f(x')$, and since $f$ is injective, we get that $\psi$ is injective.\\
For surjectivity of $\psi$, let $w\in{S^{n-1}}$. Then consider the ray $y+tw;t\in[0,\infty)$. This ray will intersect the sphere in exactly one point, say $p=y+t_0w;t_0\in{(0,2)}$. Now  using the fact that $f$ is surjective, we get that there is a point $x_0$ such that $f(x_0)=p$. That is $f(x_0)-y=t_0w$, therefore we get that $$\frac{f(x_0)-y}{\left|f(x_0)-y\right|}=\frac{t_0w}{|t_0w|}=w,$$ thus we that $f(x_0)=w$ and hence $\psi$ is surjective.\\
Now, consider $\psi^{-1} \circ \Psi: \overline{B^{n}}\rightarrow S^{n-1}$,
for $x\in {S^{n-1}}$,
\begin{align*}
\psi^{-1} \circ \Psi (x)&=\psi^{-1} \circ \psi (x) \ \mbox{(since}\  x\in {S^{n-1}} \ \mbox{and} \  \psi \  \mbox{is restriction of} \  \Psi \ \mbox{ to} \ S^{n-1}) \\
&=x.
\end{align*}
That is  $\psi^{-1} \circ \Psi$ fixes sphere pointwise. Which is a contradiction to the result (\cite{massey}, Prop.III 2.3). Thus we conclude that if $f$ maps the sphere homeomorphically onto itself, then the harmonic extension $F$ of $f$ maps the unit ball onto itself.
\end{proof}
\end{lem}
\begin{lem}
Let $\displaystyle{0\leq \phi \leq \pi}$, then
$$g_{\phi}(\pi-\theta)=\pi-g_{\phi}(\theta)\ \mbox{for} \ 0\leq \theta \leq \pi.$$
\begin{proof}
\textbf{Case 1}: When $0\leq \phi \leq \frac{\pi}{2}$\\
\underline{Subcase 1}: When $0\leq \theta \leq \frac{\pi}{2}$, which implies that
$\frac{\pi}{2}\leq \pi -\theta \leq \pi$. And therefore by definition of $g_{\phi}$, we get that 
$$g_{\phi}(\pi-\theta)=\pi -g_{\phi}(\theta).$$
\underline{Subcase 2}: When $\frac{\pi}{2} \leq \theta \leq \pi$.
Then by definiton of $g_{\phi}$,
$g_{\phi}(\theta)=\pi-g_{\phi}(\pi-\theta)$.\\
$\mbox{That is} \ $ $$g_{\phi}(\pi-\theta)=\pi-g_{\phi}(\theta).$$
\textbf{Case 2}: When $\frac{\pi}{2}\leq \phi \leq \pi$\\
\underline{Subcase 1}: When $0\leq \theta \leq \frac{\pi}{2},$ which implies that $\frac{\pi}{2}\leq \pi-\theta \leq \pi$. And therefore by definition of $g_{\phi}$,
$$g_{\phi}(\pi-\theta)=\pi-g_{\phi}(\phi-\phi+\theta)
=\pi-g_{\phi}(\theta).$$
\underline{Subcase 2}: When $\frac{\pi}{2}\leq \theta \leq \pi$
and thus again by definition of $g_{\phi}$, it follows that
$$g_{\phi}(\pi-\theta)=\pi-g_{\phi}(\theta).$$
\end{proof}
\end{lem}
\begin{lem}
Let $0\leq \phi \leq \pi$, then
 $$g_{\phi}(3\pi-\theta)=3\pi-g_{\phi}(\theta)\ \mbox{for }\ \pi \leq \theta \leq 2\pi.$$
\begin{proof}
First we prove that $0\leq \phi \leq \pi$,  $\pi \leq \theta \leq 2\pi$
implies $g_{\phi}(2\pi-\theta)=2\pi-g_{\phi}(\theta)$.\\
We know 
\begin{align*}
g_{\phi}(\theta) &=\pi+g_{\phi}(\theta-\pi)\\
&=\pi+\pi-g_{\phi}(\pi-\theta +\pi) \ \mbox{(since} \  0\leq \theta -\pi\leq \pi)\\
&=2\pi-g_{\phi}(\theta).
\end{align*}
Also note that, $\pi\leq \theta \leq 2\pi$ implies that
$3\pi-2\pi \leq 3\pi-\theta \leq 3\pi -\pi$, which further implies that
$\pi \leq 3\pi-\theta \leq 2\pi$.\\
Therefore,
\begin{align*}
g_{\phi}(3\pi-\theta)&=\pi+g_{\phi}(3\pi -\theta -\pi) \ \mbox{( by definition of} \  g_{\phi})\\
&=\pi+g_{\phi}(2\pi-\theta)\\
&=\pi+2\pi-g_{\phi}(\theta)\\
&=3\pi-g_{\phi}(\theta).
\end{align*}
\end{proof}
\end{lem}

\begin{thm}
A self homeomorphism $f$ of $S^{2}$ exists, whose harmonic extension $F$ is not injective in $B^3$.
\begin{proof}
We want to construct a function $f$ on $\mathbb{R}^3$ such that harmonic extension $F$ of $f$ is not one-one. 
For that we will construct a function $f$ on $\mathbb{R}^3$ with the following properties:
\begin{enumerate}
\item $f$ fixes the poles (0, 0, $\pm$1).
\item $f$ is symmetric in the plane $\{x=0\}$.
$\mbox{That is} \ $ $$f_1(-x,y,z)=-f_1(x,y,z)$$
 	   $$f_2(-x,y,z)=f_2(x,y,z)$$
	   $$f_3(-x,y,z)=f_3(x,y,z).$$
\item $f$ is symmetric in the plane $\{y=0\}$.
$\mbox{That is} \ $  $$f_1(x,-y,z)=f_1(x,y,z)$$
 	   $$f_2(x,-y,z)=-f_2(x,y,z)$$
	   $$f_3(x,-y,z)=f_3(x,y,z).$$
\item $F$ is folded in $z$-direction,
that is, there is $z\in{(0,1)}$ such that
$$F_3(0,0,z)<F_3(0,0,-z).$$
\end{enumerate}

\textbf{\underline{Construction in $R^{3}$:}} We take $p>0$ as a parameter (to be fixed later).\\
Let $p>0$ be fixed,
for $0\leq \phi \leq \frac{\pi}{2}$, consider function $q$ defined as
 $$q(\phi,p)=(1-\sin \phi \cos \phi)^p.$$
We claim $0<q\leq 1$.
Since $\sin \phi \cos \phi \leq 1$ and
 $1-\sin \phi \cos \phi \neq 0$, for if $1-\sin \phi \cos \phi = 0$, then $ \sin 2\phi =2$, which is absurd.
Thus $q>0$.\\
Also we have $\sin \phi \geq 0$ and $\cos \phi \geq 0$ in $\left[0,\frac{\pi}{2}\right]$, and thus
$ -\sin \phi .\cos \phi \leq 0.$
So we get that $$1-\sin \phi \cos \phi \leq 1.$$
Thus $0<q\leq 1$. Also we have,
\begin{equation}
q(0,p)=1
\end{equation}
\begin{equation}
q\left(\frac{\pi}{2},p\right)=1.
\end{equation}
And for each $\displaystyle{0<\phi <\frac{\pi}{2}}$, as $p\to \infty$
\begin{equation}
(1-\sin \phi \cos \phi)^p\to 0,\  (\mbox{  since\ } 1-\sin \phi \cos \phi < 1).
\end{equation}
\begin{equation}
\mbox{That is,} \  \quad \mbox{for}  \quad 0<\phi <\frac{\pi}{2},\  q(\phi, p) \rightarrow 0 \quad  \mbox{as} \quad p \rightarrow \infty.
\end{equation}
Now for $\displaystyle{0\leq \phi \leq \frac{\pi}{2}}$, define
\[g_{\phi}(\theta)=
\begin{cases}
\displaystyle{\frac{\pi}{2}\left(\frac{2\theta}{\pi}\right)^{q\left(\phi,p\right)}};\ \mbox{ if} \quad 0\leq \theta \leq \frac{\pi}{2}\\
\pi-g_{\phi}\left(\pi -\theta\right); \ \mbox{if} \quad \frac{\pi}{2}\leq \theta \leq \pi \\
\pi +g_{\phi}\left(\theta -\pi \right); \ \mbox{if} \quad \pi \leq \theta \leq 2\pi

\end{cases}.
\]
For $\phi=0$, we have
\begin{align*}
g_{o}(\theta)&=
\begin{cases}
\displaystyle{\frac{\pi}{2}\left(\frac{2\theta}{\pi}\right)^{q(0,p)}}; \ \mbox{if} \quad 0\leq \theta \leq \frac{\pi}{2}\\
\pi-g_{0}\left(\pi -\theta\right); \ \mbox{if} \quad \frac{\pi}{2}\leq \theta \leq \pi \\
\pi +g_{0}\left(\theta -\pi \right); \ \mbox{if} \quad \pi \leq \theta \leq 2\pi
\end{cases}\\
&=
\begin{cases}
\displaystyle{\frac{\pi}{2}\left(\frac{2\theta}{\pi}\right)}; \ \mbox{if} \quad 0\leq \theta \leq \frac{\pi}{2}\\
\pi-g_{0}\left(\pi -\theta\right);\  \mbox{if} \quad \frac{\pi}{2}\leq \theta \leq \pi \\
\pi +g_{0}\left(\theta -\pi \right);\  \mbox{if} \quad \pi \leq \theta \leq 2\pi
\end{cases}\\
&=
\begin{cases}
\theta; \quad \mbox{if} \quad 0\leq \theta \leq \frac{\pi}{2}\\
\pi-g_{0}\left(\pi -\theta\right);\  \mbox{if} \quad \frac{\pi}{2}\leq \theta \leq \pi \\
\pi +g_{0}\left(\theta -\pi \right);\  \mbox{if} \quad \pi \leq \theta \leq 2\pi
\end{cases}\\
&=
\begin{cases}
\theta; \quad \mbox{if} \quad 0\leq \theta \leq \frac{\pi}{2}\\
\pi-\left(\pi -\theta\right); \ \mbox{if} \quad \frac{\pi}{2}\leq \theta \leq \pi \\
\pi +g_{0}\left(\theta -\pi \right); \ \mbox{if} \quad \pi \leq \theta \leq 2\pi
\end{cases}\\
&=
\begin{cases}
\theta; \quad \mbox{if} \quad 0\leq \theta \leq \frac{\pi}{2}\\
\theta; \quad \mbox{if} \quad \frac{\pi}{2}\leq \theta \leq \pi \\
\pi +g_{0}\left(\theta -\pi \right);\  \mbox{if} \quad \pi \leq \theta \leq 2\pi
\end{cases}\\
&=
\begin{cases}
\theta; \quad \mbox{if} \quad 0\leq \theta \leq \frac{\pi}{2}\\
\theta; \quad \mbox{if} \quad \frac{\pi}{2}\leq \theta \leq \pi \\
\theta; \quad \mbox{if} \quad \pi \leq \theta \leq 2\pi.
\end{cases}
\end{align*}
So we get that, $$g_0=id.$$
Similarly, it can be shown that $$g_{\frac{\pi}{2}}=id.$$
For $\displaystyle{0<\phi <\frac{\pi}{2}}$,
let $0<\theta <\frac{\pi}{2}$, which implies that
\begin{equation} 0<\frac{2\theta}{\pi}<1.
\end{equation}
And in this interval, from definition of $g$, we have $$g_{\phi}(\theta)=\frac{\pi}{2}\left(\frac{2\theta}{\pi}\right)^{q(\phi,p)}.$$
Using (5.5) and (5.6), we get that  
\begin{equation}
\mbox{as} \quad p \rightarrow \infty, \ g_{\phi}(\theta) \quad \mbox{varies near } \quad \frac{\pi}{2}.
\end{equation}
Now consider the case when $\frac{\pi}{2}\leq \theta \leq \pi$, which implies that
\begin{equation}
 0\leq \pi -\theta \leq \frac{\pi}{2}.
\end{equation} 
In this case, again by definition of $g$ $$g_{\phi}(\theta)=\pi-g_{\phi}(\pi -\theta).$$
Therefore by $(5.7)$ and $(5.8)$ we get that 
\begin{equation}
g_{\phi}(\theta) \quad \mbox{varies near} \quad \frac{\pi}{2} \quad \mbox{when} \quad \frac{\pi}{2}\leq \theta \leq \pi.
\end{equation}
Now when $\pi \leq \theta \leq 2\pi,$
$\ \mbox{that is} \  0\leq \theta -\pi \leq \pi$,
 $$g_{\phi}(\theta)=\pi +g_{\phi}(\theta -\pi ).$$
So by (5.7) and (5.9) we get that 
\begin{equation}
g_{\phi}(\theta) \ \mbox{varies near} \quad 3\frac{\pi}{2} \quad \mbox{when} \quad \pi \leq \theta \leq 2\pi.
\end{equation}
We will now show that $g_{\phi}$ is a homeomorphism on interval $\left[0,\frac{\pi}{2}\right]$.\\
Clearly, $g_{\phi}$ is continuous, since if we take a sequence ${\theta}_{n} \to \theta$,\\
then $$\frac{2{\theta}_{n}}{\pi} \rightarrow \frac{2\theta}{\pi} $$
$$ \left(\frac{2{\theta}_{n}}{\pi}\right)^{q\left(\phi,p\right)} \rightarrow \left(\frac{2\theta}{\pi}\right)^{q\left(\phi,p\right)}$$
$$ \frac{\pi}{2}\left(\frac{2{\theta}_{n}}{\pi}\right)^{q\left(\phi,p\right)} \rightarrow \frac{\pi}{2}\left(\frac{2\theta}{\pi}\right)^{q\left(\phi,p\right)}$$
$$ g_{\phi}\left({\theta }_{n}\right) \rightarrow g_{\phi}\left(\theta\right),$$
and thus we get that $ g_{\phi}$ is continuous.\\
Also $g_{\phi}$ is one-one, since if
$$g_{\phi}({\theta}_1)=g_{\phi}({\theta}_2)\  \mbox{for some}\ 0\leq {\theta}_1 , {\theta}_2 \leq \frac{\pi}{2},$$
which implies that $$\frac{\pi}{2} \left(\frac{2{\theta}_1}{\pi}\right)^{q\left(\phi,p\right)}=\frac{\pi}{2} \left(\frac{2{\theta}_2}{\pi}\right)^{q\left(\phi,p\right)},$$
which further implies that $$\left(\frac{2{\theta}_1}{\pi}\right)^{q\left(\phi,p\right)}=\left(\frac{2{\theta}_2}{\pi}\right)^{q\left(\phi,p\right)},$$ thus $\theta_1=\theta_2$ and hence we get that $g_{\phi}$ is one-one.\\

Now, using that the translation of a homeomorphism is a homeomorphism, we get that 
\begin{equation}
\{g_{\phi}: \phi \in \left[0,\frac{\pi}{2}\right]\} \quad \mbox{is a collection of homeomorphisms} \quad \mbox{on} \ \left[0,2\pi\right].
\end{equation}
Now for $\phi \in {\left[\frac{\pi}{2},\pi\right]}$, define
\[g_{\phi}(\theta)=
\begin{cases}
\frac{\pi}{2}-g_{\pi-\phi}\left(\frac{\pi}{2}-\theta\right); \ \mbox{if} \quad 0\leq \theta \leq \frac{\pi}{2}\\
\pi-g_{\phi}\left(\pi -\theta\right); \ \mbox{if} \quad \frac{\pi}{2}\leq \theta \leq \pi \\
\pi +g_{\phi}\left(\theta -\pi \right); \ \mbox{if} \quad \pi \leq \theta \leq 2\pi
\end{cases}.
\]
We note that for $\phi \in \left[\frac{\pi}{2},\pi\right]$, $\pi - \phi \in {\left[0,\frac{\pi}{2}\right]}$, and thus by previous part, we get that $g_{\phi}$ is a homeomorphism in this case as well.\\
And by eq. (5.7), (5.9) and (5.10), we get that 
\begin{equation}
g_{\phi} \quad \mbox{varies near 0},\  \pi, \ 2\pi \quad \mbox{for} \quad \phi \in {\left[\frac{\pi}{2},\pi\right]}.
\end{equation}
Now, using all this we define a function $g$ on $\left[0,\pi\right]\times \left[0,2\pi\right]$ as 
$$g(\phi,\theta)=\left(\phi, g_{\phi}\left(\theta\right)\right);\phi \in {\left[\frac{\pi}{2},\pi\right]}, \theta \in{\left[0,2\pi \right]}.$$
Notice that each component is a homeomorphism by previous parts, and thus g is a self homeomorphism of $\left[0,\pi\right]\times \left[0,2\pi\right]$.\\
Clearly, g preserves line of lattitude $\{\phi =\mbox{constant}\}$.
Also g fixes boundary of $\left[0,\pi\right]\times \left[0,2\pi\right]$.
Any point on the boundary of $\left[0,\pi\right]\times \left[0,2\pi\right]$ is given by 
\begin{eqnarray}
(\phi,0);0\leq \phi \leq \pi\\
(0,\theta);0\leq \theta \leq 2\pi\\
(\phi,2\pi);0\leq \phi \leq \pi\\
(\pi,\theta);0\leq \theta \leq 2\pi
\end{eqnarray}
When $\displaystyle{0\leq \phi \leq \pi}$ and $\theta=0$. Now we consider two cases:\\
\underline{Case 1}: When $\displaystyle{0\leq \phi \leq \frac{\pi}{2}}$
\begin{align*}
g(\phi,0)&=\left(\phi,g_{\phi}(0)\right)\\
&=\left(\phi,0\right).
\end{align*}
\underline{Case 2}:When $\displaystyle{\frac{\pi}{2}\leq \phi \leq \pi}$
\begin{align*}
g(\phi,0)&=\left(\phi,g_{\phi}(0)\right)\\
&=\left(\phi,\frac{\pi}{2}-g_{\pi-\phi}\left(\frac{\pi}{2}-0\right)\right)\\
&=\left(\phi,\frac{\pi}{2}-g_{\pi-\phi}\left(\frac{\pi}{2}\right)\right)\\
&=\left(\phi,\frac{\pi}{2}-\left(\frac{\pi}{2}\right)\left(\frac{2\pi}{2\pi }\right)^{q\left(\pi-\phi,p\right)}\right)\\
&=(\phi,0).
\end{align*}
When $\phi=0$ and $\displaystyle{0\leq \theta \leq 2\pi}$. Now we consider three cases: \\
\underline{Case 1}: When $\displaystyle{0\leq \theta \leq \frac{\pi}{2}}$,
\begin{align*}
g(0,\theta)&=\left(0,\frac{\pi}{2}\left(\frac{2\theta }{\pi}\right)^{q(0,p)}\right)\\
&=\left(0,\frac{\pi}{2}\left(\frac{2\theta }{\pi}\right)\right)\\
&=(0,\theta).
\end{align*}
\underline{Case 2}: When $\displaystyle{\frac{\pi}{2} \leq \theta \leq \pi}, $
\begin{align*}
g(0,\theta )&=\left(0,\pi -g_{0}\left(\pi -\theta\right)\right)\\
&=\left(0,\pi -\frac{\pi}{2}\left(\frac{2\left(\pi-\theta\right)}{\pi}\right)^{q(0,p)}\right)\\
&=(0,\pi-\pi+\theta)\\
&=(0,\theta).
\end{align*}
\underline{Case 3}: When $\pi \leq \theta \leq 2\pi$,
then 
\begin{align*}
g(0,\theta)&=\left(0,\pi+g_{0}\left(\theta -\pi\right)\right)\\
&=\left(0,\pi+\theta -\pi\right)\ (\mbox{since $g_0$=id})\\
&=(0,\theta).
\end{align*} 
When $0\leq \phi \leq \pi$ and $\theta =2\pi$. Now we consider two cases:\\
\underline{Case 1}: When $\displaystyle{0 \leq \phi \leq \frac{\pi}{2}}$
\begin{align*}
g(\phi,2\pi)&=\left(\phi,g_{\phi}\left(2\pi\right)\right)\\
&=\left(\phi,\pi+g_{\phi}\left(2\pi -\pi\right)\right)\\
&=\left(\phi,\pi +g_{\phi}\left(\pi\right)\right)\\
&=\left(\phi,\pi +\pi -g_{\phi}\left(\pi -\pi\right)\right)\\
&=\left(\phi,\pi+\pi-g_{\phi}(0)\right)\\
&=(\phi,2\pi).
\end{align*}
\underline{Case 2}: When $\displaystyle{\frac{\pi}{2} \leq \phi \leq \pi }$\\
\begin{align*}
g(\phi,2\pi)&=\left(\phi,g_{\phi}\left(2\pi\right)\right)\\
&=\left(\phi,\pi+g_{\phi}\left(2\pi -\pi\right)\right)\\
&=\left(\phi,\pi+\pi-g_{\phi}(0)\right)\\
&=\left(\phi,2 \pi- \frac{\pi}{2}+g_{\pi-\phi} \left(\frac{\pi}{2} \right)\right)\\
&=\left(\phi,2\pi \right)
\end{align*}
When $0\leq \theta \leq 2\pi$ and $\phi=\pi$. Now we consider three cases:\\
\underline{Case 1}: When $0\leq \theta \leq \frac{\pi}{2}$
\begin{align*}
g(\pi,\theta)&=\left(\pi,g_{\pi}(\theta)\right)\\
&=\left(\pi,\frac{\pi}{2}-g_{\pi -\pi}\left(\frac{\pi}{2}-\theta\right)\right)\\
&=\left(\pi,\frac{\pi}{2}-g_0\left(\frac{\pi}{2}-\theta\right)\right)\\
&=\left(\pi,\frac{\pi}{2}-\frac{\pi}{2}+\theta\right)\\
&=(\pi,\theta).
\end{align*}
\underline{Case 2}: When $\displaystyle{\frac{\pi}{2} \leq \theta \leq \pi} $
\begin{align*}
g(\pi,\theta)&=\left(\pi,\pi -g_{\pi}\left(\pi-\theta\right)\right)\\
&=\left(\pi,\pi-\frac{\pi}{2}+g_{\pi-\pi}\left(\frac{\pi}{2}-\pi +\theta\right)\right)\\
&=\left(\pi,\frac{\pi}{2}+g_0\left(-\frac{\pi}{2}+\theta\right)\right)\\
&=\left(\pi,\frac{\pi}{2}-\frac{\pi}{2}+\theta\right)\\
&=(\pi,\theta).
\end{align*}
\underline{Case 3}: When, $\pi \leq \theta \leq 2\pi$
\begin{align*}
g(\pi,\theta)&=(\pi,g_{\pi}(\theta))\\
&=\left(\pi,\pi+g_{\pi}\left(\theta -\pi\right)\right),
\end{align*}
we will consider two subcases here:\\
\textbf{Subcase 1}: When $\displaystyle{0\leq \theta-\pi \leq \frac{\pi}{2}}$, then
\begin{align*}
g(\pi,\theta)&=\left(\pi,\pi+g_{\pi}\left(\theta -\pi\right)\right)\\
&=\left(\pi,\pi+\frac{\pi}{2}-g_0\left(\frac{\pi}{2}-\theta-\pi\right)\right)\\
&=\left(\pi,\theta\right).
\end{align*}
\textbf{Subcase 2}: When $\displaystyle{\frac{\pi}{2}\leq \theta-\pi \leq \pi}$, then
\begin{align*}
g(\pi,\theta)&=\left(\pi,\pi+\pi-g_{\pi}\left(\pi-\theta+\pi\right)\right)\\
&=\left(\pi,2\pi-g_{\pi}(2\pi-\theta)\right)\\
&=\left(\pi,2\pi+\frac{\pi}{2}-g_0\left(\frac{\pi}{2}-2\pi+\theta\right)\right)\\
&=\left(\pi,\theta\right).
\end{align*}
Hence in all the cases we have seen that, $g$ fixes boundary of $[0,\pi]\times  [0,2\pi]$.
Now for $0\leq \phi \leq \pi$, we define another function $h_{\phi}$ as
\[h_{\phi}(\theta)=
\begin{cases}
\pi \left(\frac{\phi}{\pi}\right)^{1+p\left(\pi-4\theta\right)};\  \mbox{if} \ 0\leq \theta \leq \frac{\pi}{4}\\
\pi-h_{\pi-\phi}(\frac{\pi}{2}-\theta); \ \mbox{if} \ \displaystyle{\frac{\pi}{4}\leq \theta \leq \frac{\pi}{2}}\\
h_{\phi}\left(\pi-\theta\right);\ \mbox{if} \ \displaystyle{\frac{\pi}{2} \leq \theta \leq {\pi}}\\
h_{\phi(\theta-\pi); \ \mbox{if}\ \displaystyle{\pi \leq \theta \leq 2\pi}}
\end{cases}.
\]
Then as in the case of $g_{\phi}$ it can be easily checked that $h_{\frac{\pi}{4}}=id$, $h_{\phi}\left(\frac{3\pi}{4}\right)=\phi$ and $h_{\phi}(0)=h_{\phi}(2\pi)$. Also for all $h_0\equiv 0$ and $h_{\pi}\equiv \pi.$ And using all this it can be easily verified that the map $h:\left[0,\pi\right]\times \left[0,2\pi\right]\to \left[0,\pi \right] \times \left[0,2\pi \right]$ defined as $$h(\phi,\theta)=(h_{\phi}(\theta),\theta)$$ is a homeomorphism and fixes the lines $\{\phi=0 \}$ and $\{\phi=\pi \}.$
Now we define a map $k$ on $\left[0,\pi\right]\times\left[0,2\pi\right]$ as
$$k=h\circ g.$$
\begin{align*}
\mbox{That is} \ k(\phi,\theta)&=h(g(\phi,\theta))\\
&=h(\phi,g_{\phi}(\theta))\\
&=\left(h_{\phi}(g_{\phi}\left(\theta\right)\right),g_{\phi}(\theta)).
\end{align*}
Notice that $k$ fixes pointwise lines $\{\phi=0\}$ and $\{\theta=\phi\}$.
Since any point on the line $\{\phi=0\}$ is given by
$\{(0,\theta):0\leq \theta \leq 2\pi\}$
and \begin{align*}
k(0,\theta)&=(h_{0}(g_{0}(\theta)),g_{0}(\theta))\\
&=(h_0(\theta),\theta) \ \mbox{(since }\ g_0=id)\\
&=(0,\theta) \ \mbox{ (since }\ h_0=0).\\
\end{align*}
Also, any point on line $\{\theta=\phi\}$ is given by 
$\{(\pi,\theta):0\leq \theta \leq 2\pi\}$
and 
\begin{align*}
k(\pi,\theta)&=\left(h_{\pi}(g_{\pi}\left(\theta\right)\right),g_{\pi}(\theta))\\
&=\left(h_{\pi}\left(\theta\right),\theta\right)\  \mbox{(since} \  g_{\pi}=id )\\
&=(\pi,\theta)\ \mbox{(since} \  h_{\pi}\equiv \pi).
\end{align*}
Now in order to define $f$ on $S^2$,
take $(x,y,z)\in{S^2}$, therefore
\begin{align*}
&x=\sin \phi \cos \theta\\
&y=\sin \phi \sin \theta\\
&z=\cos \phi.
\end{align*}
Define $f$ on $S^2$ as 
$$f(x,y,z)=f(\sin \phi \cos \theta,\sin \phi \sin \theta,\cos \phi)=
(\sin \phi ' \cos \theta ',\sin \phi ' \cos \theta ',\cos \phi ')
=(x',y',z'),$$
where $k(\phi,\theta)=\left(h_{\phi}\left(g_{\phi}\left(\theta\right)\right),g_{\phi}(\theta)\right)
=(\phi ',\theta ')$.\\
Putting $\phi =0$, we see that
\begin{align*}
k(0,\theta)&=(h_{0}(g_{0}(\theta)),g_0(\theta))\\
&=(0,g_0(\theta))\ \mbox{ \ (since} \  h_0\equiv 0)\\
&=(0,\theta)\ \mbox{(since} \  g_{0}=id).
\end{align*}
So we get that 
$$f(0,0,1)=(0,0,1).$$
$\mbox{That is}, \ f$ fixes north pole.
Now putting $\phi =\pi$, we see that
\begin{align*}
k(\pi,\theta)&=\left(h_{\pi}(g_{\pi}\left(\theta\right)\right),g_{\pi}\left(\theta\right))\\
&=\left(\pi,g_{\pi}\left(\theta\right)\right)\ \mbox{ (since} \  h_{\pi}\equiv \pi)\\
&=\left(\pi,\theta\right)\ \mbox{ (since} \  g_{\pi}=id)
\end{align*}
So we get that, 
$$f(0,0,-1)=(0,0,-1).$$
$\mbox{That is}, \  f$ fixes south pole.\\
Now we will show that $f$ is symmetric about the plane $\{x=0\}$.
\begin{align*}
\mbox{That is} \ &f_1(-x,y,z)=-f_1(x,y,z)\\
&f_2(-x,y,z)=f_2(x,y,z)\\
&f_3(-x,y,z)=f_3(x,y,z).
\end{align*}
We have,
\begin{align*}
f(x,y,z)&=f(\sin \phi \cos \theta,\sin \phi \sin \theta,\cos \phi)\\
&=(\sin \phi ' \cos \theta ',\sin \phi ' \cos \theta ',\cos \phi ')\\
&=(f_1(x,y,z),f_2(x,y,z),f_3(x,y,z)).
\end{align*}
\begin{eqnarray}
f_{1}(x,y,z)=\sin \phi ' \cos \theta'\\
f_2(x,y,z)=\sin \phi'  \sin \theta'\\
f_3(x,y,z)=\cos \phi'.
\end{eqnarray}
So, $f(-x,y,z)=f(-\sin \phi \cos \theta,\sin \phi \sin \theta,\cos \phi)$.\\
When $0\leq \theta \leq \pi$\\
\begin{align*}
f(-x,y,z)&=f(-\sin \phi \cos \theta,\sin \phi  \sin \theta,\cos \phi)\\
&=f(\sin \phi \cos (\pi - \theta ),\sin \phi \sin (\pi -\theta),\cos \phi)).
\end{align*}
Now, 
in order to prove symmetry about the plane $\{x=0\}$, 
consider
\begin{align*}
k(\phi,\pi -\theta)&=(h_{\phi}(g_{\phi}(\pi-\theta)),g_{\phi}(\pi-\theta))\\
&=(h_{\phi}(\pi-g_{\phi}(\theta)),\pi-g_{\phi}(\theta))\ \mbox{ (using lemma 5.2.2)}\\
&=(h_{\phi}(\pi -\theta '),\pi -\theta ')\ \mbox{(where} \  k(\phi,\theta)=(\phi ',\theta '))\\
&=(h_{\phi}(\theta '),\pi -\theta ')\ \\
&=(\phi ',\pi -\theta '),
\end{align*}
we obtain that
\begin{align*}
f_1(-x,y,z)&=\sin \phi ' \cos (\pi -\theta ')\\
&=-\sin \phi ' \cos \theta '\\
&=-f_1(x,y,z),
\end{align*}
\begin{align*}
f_2(-x,y,z)&=\sin \phi '  \sin (\pi -\theta ')\\
&=f_2(x,y,z),
\end{align*}
 and       	$\	\	\	\	\	\	\	\	\	\ 	\	 \	 \	 \	 \	 \	 \	 \	 \	 \	 \	 \	\	\	\	\	\	\	\	\	\	f_3(-x,y,z)=\cos \phi '.$\\
Thus we get that $f$ is symmetric with respect to plane $\{x=0\}$. In case when $\pi \leq \theta \leq 2\pi$, by considering angle $3\pi-\theta$ instead of angle $\pi-\theta$ and using lemma $5.2.3$, symmetry about the plane $\{x=0\}$ can be obtained.
Now we will show that $f$ is symmetric with respect to plane $\{y=0\}$.
\begin{eqnarray}
\mbox{That is} \ f_1(x,-y,z)=f_1(x,y,z)\\
f_2(x,-y,z)=-f_2(x,y,z)\\
f_3(x,-y,z)=f_3(x,y,z).
\end{eqnarray}
\begin{align*}
f(x,-y,z)&=f\left(\sin \phi \cos \theta,-\sin \phi \sin \theta,\cos \phi\right)\\
&=f\left(\sin \phi \cos (2\pi -\theta),\sin \phi \sin \left(2\pi - \theta\right),\cos \phi\right).
\end{align*}
Now in order to prove symmetry about the plane $\{y=0\}$,\\
consider
\begin{align*}
k(\phi, 2\pi -\theta)&=(h_{\phi}(g_{\phi}(2\pi -\theta)),g_{\phi}(2\pi -\theta))\\
&=(h_{\phi}(2\pi-g_{\phi}(\theta)),2\pi-g_{\phi}(\theta))\mbox{ (using lemma 5.2.2)}\\
&=(h_{\phi}(2\pi -\theta '),2\pi -\theta ')\\
&=(h_{\phi}(\theta '),2\pi -\theta ')\ \mbox{(since} \ h_{\phi}(2\pi-\theta)=h_{
\phi}(\theta))\\
&=(\phi ',2\pi -\theta ').
\end{align*}
Therefore \begin{align*}
f_1(x,-y,z)&=\sin \phi ' \cos (2\pi -\theta ')\\
&=\sin \phi ' \cos \theta '\\
&=f_1(x,y,z),
\end{align*}
\begin{align*}
f_2(x,y,z)&=\sin \phi ' \sin (2\pi -\theta ')\\
&=-\sin \phi ' \sin \theta '\\
&=-f_2(x,y,z),
\end{align*}
and $f_3(x,-y,z)=\cos \phi '=f_3(x,y,z)$.
Thus $f$ is symmetric with respect to plane $\{y=0\}$.\\
Let $F$ be extension of $f$. We claim that 
there is $z$ , $0<z<1$ such that
$$F_3(0,0,z)<F_3(0,0,-z).$$
\textbf{\underline{Step-1}}
Let $(x,y,z)\in{S^2}$ such that $x,y \neq 0$ and $z>0$.
We claim that $f_3(x,y,z)\rightarrow -1$ as $p \rightarrow \infty$.
$\mbox{That is} \  \cos \phi ' \rightarrow -1$ as $p \rightarrow \infty$. So if we show that $\phi ' \rightarrow \pi$ as $p\rightarrow \infty$, then we are done.
Since $f$ is symmetric with respect to planes $\{x=0\}$ and $\{y=0\}$, without loss of generality, we may assume that $x,y>0$.\\
$\mbox{That is}, \ $ without loss of generality we may assume that $\displaystyle{0<\phi <\frac{\pi}{2}}$ and $\displaystyle{0<\theta <\frac{\pi}{2}}$.
First of all notice that, in this case 
$g_{\phi}(\theta)=\frac{\pi}{2}\left(\frac{2\theta}{\pi}\right)^{q(\phi ,p)}$,
and since $q(\phi ,p)\rightarrow 0$ as $p\rightarrow \infty$, we get that
\begin{equation}
g_{\phi}(\theta) \rightarrow \frac{\pi}{2} \quad as \quad p \rightarrow \infty.
\end{equation}
Now consider
\begin{align*}
\phi '&=h_{\phi}(g_{\phi}(\theta))\\
&=\pi -h_{\pi -\phi}\left(\left(\frac{\pi}{2}-g_{\phi}\left(\theta\right)\right)\right)\ \mbox{ (using definition of } \ h_{\phi}\ \mbox{ and  equation (5.23))}\\
&=\pi-\pi\left(\frac{\pi -\phi}{\pi}\right)^{1+p\left(\pi -4\left(\frac{\pi}{2}-g_{\phi}\left(\theta\right)\right)\right)}\ \mbox{(using definition of} \ h_{\phi} \ \mbox{and equation (5.23))},
\end{align*}
$\	\	\	\	\	\	\	\	\	\	\	\	\to \pi \  $ as $p \rightarrow \infty$. \\

(here we have used that $\displaystyle{0\leq \frac{\pi-\phi}{\pi}\leq 1}$ for $\displaystyle{0<\phi <\frac{\pi}{2}}$ and $\displaystyle{g_{\phi}(\theta) \rightarrow \frac{\pi}{2}}$ as $p\to \infty$)\\
Therefore we get that  $$f_3(x,y,z) \rightarrow \cos \phi '=-1\ \mbox{ as}\  p \rightarrow \infty.$$
$\mbox{That is,} \  f$ drags the northern hemisphere towards the south pole.\\
\underline{\textbf{Step -2}}\\
Consider case when  $x,y >0$ and $z<0$.
In this case , we have
$\frac{\pi}{2}<\phi < \pi$ and $0<\theta < \frac{\pi}{2}.$
Now, \begin{align*}
g_{\phi}\left(\theta\right)&=\frac{\pi}{2}-g_{\pi -\phi}\left(\frac{\pi}{2}-\theta\right)\\
&=\frac{\pi}{2}-\left(\frac{\pi}{2}\right)\left(\frac{2\left(\frac{\pi}{2}-\theta\right)}{\pi}\right)^{q\left(\pi -\phi,p\right)}.
\end{align*}
$\rightarrow 0$ as $p \rightarrow \infty$.\\
(since $\displaystyle{0<\pi -\phi < \frac{\pi}{2}}$ and therefore $q(\pi-\phi,p)\to 0$ as $p\to \infty$)\\
Now, \begin{align*}
\phi ' &=h_{\phi}(g_{\phi}(\theta))\\
&=\pi \left(\frac{\phi}{\pi}\right)^{1+p\left(\pi -4g_{\phi}(\theta)\right)}\\
&\rightarrow 0 \  \mbox{as} \  p \rightarrow \infty.
\end{align*}
Since $0<\phi <\pi \Rightarrow \frac{\phi}{\pi} <1$,
so we get that $\phi ' \rightarrow 0$ as $p \rightarrow \infty$.\\
That is $f_3(x,y,z)\rightarrow \cos \phi ' =1$ as $p \rightarrow \infty.$
Thus we see that $f$ drags the southern hemisphere towards north pole.\\
Since $f$ is symmetric in the planes $\{x=0\}$ and $\{y=0\}$, so for any point $(0,0,z)$ on $z-$axis
\begin{align*}
f(0,0,z)&=(f_1(0,0,z),f_2(0,0,z),f_3(0,0,z))\\
&=(-f_1(0,0,z),-f_2(0,0,z),f_3(0,0,z)).
\end{align*}
Thus $f(0,0,z)=(0,0,f_3(0,0,z))$,
$\mbox{that is} \ $ $f$ maps $z$-axis onto itself.
Since $f$ is symmetric in planes $\{x=0\}$ and $\{y=0\}$, therefore $F$ is symmetric in these planes, and consequently $F$ maps $z$-axis onto itself.\\
Now write $u$ for the harmonic function in $B^3$ which equals identically $-1$ northern hemisphere $\{z>0\}$ and equals $1$ on southern hemisphere $\{z<0\}$ of $S^2$.
Then for any $z\in{(0,1)}$, we have
$$u(0,0,z)=-1<0<1=u(0,0,-z),$$
and since $F_3\rightarrow u$ pointwise as $p\rightarrow \infty$.
So, $F_3(0,0,z)<F_3(0,0,-z)$, which gives that $F_3$ is not injective
and since $F$ maps $z$-axis onto itself, we get that $F$ is not injective.
\end{proof}
\end{thm}
\begin{thm}
There exists a self homoeomorphism $f^*$ of $S^{n-1}$ exists whose harmonic extension $F^*$ is not injective in $B^n$.
\begin{proof}
Denote self homeomorphism of $S^2$ by $f$ and its harmonic extension on $B^3$ by $F$ constructed in the previous theorem.\\
Let $n\geq 4$, for $x=(x_1,x_2,\ldots , x_n)\in{S^{n-1}}$, consider\\
$$R(x)=\sqrt{|x_{n-2}|^2+|x_{n-1}|^2+|x_n|^2}=\sqrt{1-|x_1|^2-\ldots -|x_{n-3}|^2}.$$
Consider $f^*:S^{n-1}\to S^{n-1}$ given by \\
\[ f^*(x)=
\begin{cases}
\left(x_1,\ldots , x_{n-3},0,0,0\right); \	\	\	\	\	\	\	\	\	\	\ \ \ \ \ \ \ \ \ \ \ \ \ \ \ \ \ \ \ \ \ \mbox{if} \  R(x)=0\\
\left(x_1,\ldots , x_{n-3},R(x)\left(f\left(\frac{x_{n-2}}{R(x)},\frac{x_{n-1}}{R(x)},\frac{x_n}{R(x)}\right)\right)\right); \ \mbox{if} \ R(x)>0
\end{cases}.
\]
Consider north pole $x_0=(0,0,\ldots , 0,0,1)$ of $S^{n-1}$,
then $R(x_0)=1>0$, and since $f$ fixes north pole of $S^2$, we get that
$$f^*(0,0,\ldots ,0,0,1)=(0,0,\ldots , 0,0,1).$$
Similarly, $f^*$ fixes south pole, and since $f^*$ is idednity on first $n-3$ co-ordinates and using the symmetry of map $f$ in planes $\{x_{n-2}=0\}$ and $\{x_{n-1}=0\}$ in $S^2$, it can be easily checked that $f^*$ is symmetric in planes $\{x_{i}\}$, $i=1,2,\ldots , n-1$ and hence so is $F^*$.\\
Write $u^*$ for the harmonic function which equals identically $-R(x)$ on the northen hemisphere $\{x\in{S^n}:x_n>0\}$ and equals $R(x)$ on southern hemisphere $\{x\in{S^n}:x_n<0\}$ of $B^n$.
Then for any $x_n\in{(0,1)}$,
$$u^*(0,0,\ldots , x_n)=-1<0<1=u(0,0,\ldots , -x_n),$$
and by the Poisson formula $F_n^*\rightarrow u^*$, and hence 
$$F_n^*(0,0,\ldots , x_n)<F_n^*(0,0,\ldots , -x_n).$$
Therefore $F_n^*$ is not injective and hence $F^*$ is not injective, since $F^*$ maps $x_n-$axis onto itself.

\end{proof}
\end{thm}

\end{document}